\documentclass{mfat}
\pagespan{266}{282}

\usepackage{mmap}
\usepackage[utf8]{inputenc}

\usepackage[all]{xy}

\newtheorem{theorem}[subsection]{Theorem}
\newtheorem{lemma}[subsection]{Lemma}
\newtheorem{sublemma}[subsubsection]{Lemma}
\newtheorem{proposition}[subsection]{Proposition}

\theoremstyle{definition}
\newtheorem{definition}[subsection]{Definition}

\theoremstyle{remark}

\newtheorem{example}[subsection]{Example}

\makeatletter \@addtoreset{subsection}{section}
\@addtoreset{equation}{section} \@addtoreset{figure}{section}
\@addtoreset{table}{section} \makeatother
\renewcommand{\theequation}{\thesection.\arabic{equation}}
\renewcommand{\thefigure}{\thesection.\arabic{figure}}
\renewcommand{\thetable}{\thesection.\arabic{table}}

\makeatletter
\newcommand\testshape{family=\f@family; series=\f@series; shape=\f@shape.}
\def\myemphInternal#1{\if n\f@shape%
\begingroup\itshape #1\endgroup\/%
\else\begingroup\bfseries #1\endgroup%
\fi}
\def\myemph{\futurelet\testchar\MaybeOptArgmyemph}
\def\MaybeOptArgmyemph{\ifx[\testchar \let\next\OptArgmyemph
                 \else \let\next\NoOptArgmyemph \fi \next}
\def\OptArgmyemph[#1]#2{\index{#1}\myemphInternal{#2}}
\def\NoOptArgmyemph#1{\myemphInternal{#1}}
\makeatother

\newcommand{\bN}{\mathbb{N}}

\newcommand{\bR}{\mathbb{R}}

\newcommand{\cU}{\mathcal{U}}

\newcommand{\Int}{\mathop{\mathrm{Int}}\nolimits}

\newcommand\Ksp{K}

\newcommand\Qsp{Q}

\newcommand\Xsp{X}
\newcommand\Ysp{Y}
\newcommand\Zsp{Z}
\newcommand\Usp{U}
\newcommand\Vsp{V}
\newcommand\Wsp{W}
\newcommand\Partition{\Delta}
\newcommand\elem{\omega}

\newcommand\prj{p}
\newcommand\bnd[1]{\mathrm{hcl}(#1)}
\newcommand\bnds[1]{\mathrm{hcl}_{S}(#1)}

\newcommand\YspecPtSet{\mathcal{V}}

\newcommand\classFol{\mathcal{F}}
\newcommand\strip{S}
\newcommand\eps{\varepsilon}
\newcommand\crosssect{\sigma}

\newcommand\trap{T}
\newcommand\strap{S}

\newcommand\Qmin{A}
\newcommand\Qmax{B}

\newcommand\iQmin{\Qmin^{o}}
\newcommand\iQmax{\Qmax^{o}}

\newcommand\aSet{\mathcal{A}}
\newcommand\bSet{\mathcal{B}}
\newcommand\ai{\alpha}

\newcommand\Jj{J}
\newcommand\Jint{K}
\newcommand\JOint{\Jint^{o}}
\newcommand\JA{K}
\newcommand\JB{L}
\newcommand\JC{M}
\newcommand\JOA{\JA^{o}}
\newcommand\JOB{\JB^{o}}

\newcommand\ahom{\kappa}
\newcommand\bhom{\lambda}
\newcommand\chom{\mu}
\newcommand\aahom{\eta_{K}}
\newcommand\bbhom{\eta_{L}}

\newcommand\Ac{A}
\newcommand\Bc{B}

\newcommand\onehalf{\tfrac{1}{2}}

\newcommand\func{\gamma}

\newcommand\dd{d}

\newcommand\roof{\mathrm{roof}}

\begin{document}

\title[Foliations with all non-closed leaves on non-compact surfaces]
    {Foliations with all non-closed leaves on non-compact surfaces}

\author{Sergiy Maksymenko}
\address{Institute of Mathematics, National Academy of Sciences of
Ukraine, 3 Teresh\-chenkivs'ka, Kyiv, 01601, Ukraine}
\email{Sergiy Maksymenko <maks@imath.kiev.ua>}

\author{Eugene Polulyakh}
\address{Institute of Mathematics, National Academy of Sciences of
Ukraine, 3 Teresh\-chenkivs'ka, Kyiv, 01601, Ukraine}
\email{Eugene Polulyakh <polulyah@imath.kiev.ua>}

\date{30/05/2016}

\begin{abstract}
Let $X$ be a connected non-compact $2$-dimensional manifold
possibly with boundary and $\Delta$ be a foliation on $X$ such
that each leaf $\omega\in\Delta$ is homeomorphic to $\mathbb{R}$
and has a trivially foliated neighborhood. Such foliations on the
plane were studied by W.~Kaplan who also gave their topological
classification. He proved that the plane splits into a family of
open strips foliated by parallel lines and glued along some
boundary intervals. However W.~Kaplan's construction depends on a
choice of those intervals, and a foliation is described in a
non-unique way. We propose a canonical cutting by open strips
which gives a uniqueness of classifying invariant. We also describe
topological types of closures of those strips under additional
assumptions on $\Delta$.
\end{abstract}

\maketitle

\section{Introduction}
Let $\Xsp$ be a $2$-dimensional manifold possibly non-connected and having a boundary, and $\Partition$ be a one-dimensional foliation on $\Xsp$.
We will say that $\Partition$ belongs to class $\classFol$ if it satisfies the following three conditions.
\begin{enumerate}
\item\label{fol:enum:leaves_are_closed}
Each leaf $\omega$ of $\Partition$ is a closed subset of $\Xsp$.

\item\label{fol:enum:bnd_consists_of_leaves}
Every connected component $\omega$ of $\partial\Xsp$ is a leaf of $\Partition$.

\item\label{fol:enum:loctriv_fol}
Let $\omega \in \Partition$ be a leaf, and $J=[0,1)$ if $\omega\subset\partial\Xsp$, and $J=(-1,1)$ otherwise.
Then there exists an open neighborhood $\Usp$ of $\omega$ and a homeomorphism $\phi:\bR\times J \to \Usp$ such that $\phi(\bR\times 0) = \omega$ and $\phi(\bR\times t)$ is a leaf of $\Partition$ for all $t\in J$, see Figure~\ref{fig:foliation_classF}.
\end{enumerate}

Roughly speaking, \emph{a $1$-dimensional foliation} $\Partition$ is a partition of $\Xsp$
 which looks like a partition of $\bR^2$ into parallel lines \emph{near each point $x\in\Xsp$}.
Then $\Partition$ belongs to class $\classFol$ whenever it looks like partition of $\bR^2$ into parallel lines
\emph{near each leaf $\omega\in\Partition$}.
In particular, each leaf of $\Partition$ is homeomorphic to $\bR$.

\begin{figure}[ht]
\begin{tabular}{ccc}
\includegraphics[height=1.3cm]{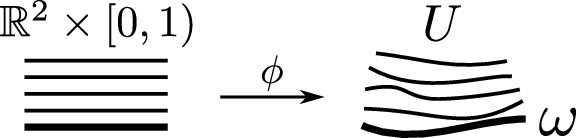} & \qquad \qquad &
\includegraphics[height=1.3cm]{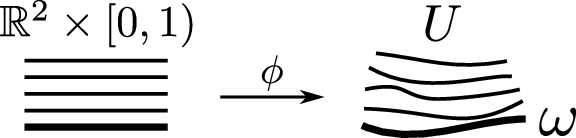} \\
(a) leaf in the boundary & &
(b) leaf in the interior
\end{tabular}
\caption{}\label{fig:foliation_classF}
\end{figure}

\vspace*{-2mm}

\begin{definition}
  Let $\Xsp_i$ be a surface with a foliation $\Partition_i$, $i=1,2$.
  Then a homeomorphism $h:\Xsp_1 \to \Xsp_2$ will be called
  \emph{foliated} if it maps leaves of $\Partition_1$ onto leaves of
  $\Partition_2$.  In this case we will also write $h:
  (\Xsp_1,\Partition_1) \to (\Xsp_2,\Partition_2)$.
\end{definition}

The aim of the present paper is to describe a topological structure of foliations belonging to class $\classFol$ up to foliated homeomorphisms, see Theorem~\ref{th:open_strips} below.
Such foliations on the plane were studied by W.~Kaplan~\cite{Kaplan:DJM:1940} and they appear as foliations by level sets of pseudoharmonic functions on $\bR^2$, see W.~Kaplan~\cite[Theorem~42]{Kaplan:DJM:1940}, W.~Boothby~\cite{Boothby:AJM_1:1951}, \cite{Boothby:AJM_2:1951}, M.~Morse and J.~Jenkins~\cite{JenkinsMorse:AJM:1952}, M.~Morse~\cite{Morse:FM:1952}.
We will improve Kaplan's construction and extend it to foliations on arbitrary surfaces.

Topological structure of singular foliations on surfaces, in
particular, foliations by orbits of flows, were studied by A.~Andronov
and L.~Pontryagin~\cite{AndronovPontryagin:DANSSSR:1937},
M.~Peixoto~\cite{Peixoto:Top:1962}, \cite{Peixoto:Top:1963},
S.~Aranson and V.~Grines~\cite{AransonGrines:MatSb:1973,
  AransonGrines:UMN:1986}, I.~Bronstein and
I.~Nikolayev~\cite{BronsteinNikolayev:TA:1997}, S.~Aranson,
E.~Zhuzhoma, and
V.~Medvedev~\cite{AransonZhuzhomaMedvedev:MatSb:1997},
L.~Plachta~\cite{Plachta:fol1:TA:2003,
  Plachta:fol2:MMFMP:2001,Plachta:fol3:MMFMP:2001}, A.~Oshemkov and
V.~Sharko~\cite{OshemkovSharko:MatSb:1998}, S.~Aranson, V.~Grines and
V.~Kaimanovich~\cite{AransonGrinesKaimanovich:JDCS:2003},
M.~Farber~\cite{Farber:AMSP:2004}, N.~Budnytska and
O.~Prishlyak~\cite{BudnytskaPryshlyak:UMJ:2009}, N.~Budnyts'ka and
T.~Rybalkina~\cite{BudnytskaRybalkina:UMJ:2012} and many others.
Results of the paper could also be applied to singular foliations
without non-closed leaves on surfaces by removing singularities.  This
will be done in subsequent papers by the authors.

\smallskip

\subsection*{Special leaves}
Suppose $\Partition$ is a foliation of class $\classFol$ on a surface $\Xsp$.
Let $\Ysp = \Xsp/\Partition$ be the space of leaves, and $\prj:\Xsp\to\Ysp$ be the corresponding quotient map.
Endow $\Ysp$ with the \emph{quotient topology}, so a subset $\Vsp\subset\Ysp$ is open if and only if its inverse $\prj^{-1}(\Vsp)$ is open in $\Xsp$.
For a subset $\Usp\subset\Xsp$ its \emph{saturation}, $S(\Usp)$, with respect to $\Partition$ is the
union of all leaves of $\Partition$ intersecting $\Usp$.
Equivalently, $S(\Usp) = \prj^{-1}(\prj(\Usp))$.

Since each leaf of $\Partition$ is a closed subset of $\Xsp$, it follows that $\Ysp$ is a $T_1$-space.
However, in general, $\Ysp$ is not a Hausdorff space.

\begin{lemma}\label{lm:classFol_prop}
If $\Partition\in\classFol$ then the projection map $\prj:\Xsp\to\Ysp$ is open.
\end{lemma}
\begin{proof}
We have to prove that for each open $\Vsp \subset \Xsp$ its saturation $S(\Vsp)$ is open as well.
Thus for each $x\in S(\Vsp)$ we should find an open saturated subset $\Wsp$ such that $x\in\Wsp=S(\Wsp) \subset S(\Vsp)$.
Let $\omega$ be the leaf containing $x$.
Put $J=[0,1)$ whenever $\omega\subset\partial\Xsp$ and $J=(-1,1)$ otherwise.
Then by definition of class $\classFol$ there exists a foliated homeomorphism $\phi:\bR\times J \to \Usp$ such that $\phi^{-1}(x)=(t,0) \in \bR\times 0$ for some $t\in\bR$.
Then $\phi^{-1}(\Vsp\cap\Usp)$ is an open neighborhood of $(t,0)$, whence there exists $\eps>0$ such that if we denote $K = J \cap (-\eps,\eps)$, then $t\times K \subset \phi^{-1}(\Vsp\cap\Usp)$.
But $K$ is open in $J$, whence $\bR\times K$ is open in $\bR\times J$.
Therefore $\phi(\bR\times K)$ is saturated and open in $\Usp$ which in turn is open in $\Xsp$.
Hence $\phi(\bR\times K)$ is open in $\Xsp$ and $x\in\phi(\bR\times K) \subset S(\Vsp)$.
Therefore $S(\Vsp)$ is open in $\Xsp$.
\end{proof}

\begin{definition}
Let $\omega$ be a leaf of $\Partition$ and $y=\prj(\omega)\in\Ysp$.
We will say that $\omega$ is a \emph{special} leaf and $y$ is a \emph{special} point of $\Ysp$  whenever $\Ysp$ is
not Hausdorff at $y$, that is $y \not = \cap_{y\in \Vsp} \overline{\Vsp}$, where $\Vsp$ runs over all open
neighborhoods of $y$.
\end{definition}

\begin{example}\label{exmp:special_leaves}\rm
Consider the foliation on $\bR^2$ shown in Figure~\ref{fig:example:special_leafs}(a).
It splits by bold leaves $\alpha$, $\beta$, $\gamma$, and $\delta$ into five ``strips'' $A$, $B$, $C$, $D$, $E$ foliated by ``parallel'' lines, see Figure~\ref{fig:example:special_leafs}(b).
Moreover, the space of leaves $\Ysp$ has the structure as in Figure~\ref{fig:example:special_leafs}(c), where bold lines correspond to strips, and thin lines just indicate that $\alpha$ belongs to the closure of $A$ and $B$, $\beta$ belongs to the closures $B$ and $C$ and so on.
In particular, $\Ysp$ looses Hausdorff property at $\alpha$, $\beta$, $\gamma$, and $\delta$.
More precisely, the subspace $\Ysp\setminus\{\alpha,\beta,\gamma,\delta\}$ is Hausdorff, however each neighborhood of $\alpha$ intersect each neighborhood of $\beta$, and the same holds for pairs $\{\beta, \gamma\}$ and $\{\gamma, \delta\}$.
Therefore the leaves $\alpha$, $\beta$, $\gamma$ and $\delta$ are special.
\end{example}

\begin{figure}[ht]
\begin{tabular}{ccccc}
\includegraphics[height=1.7cm]{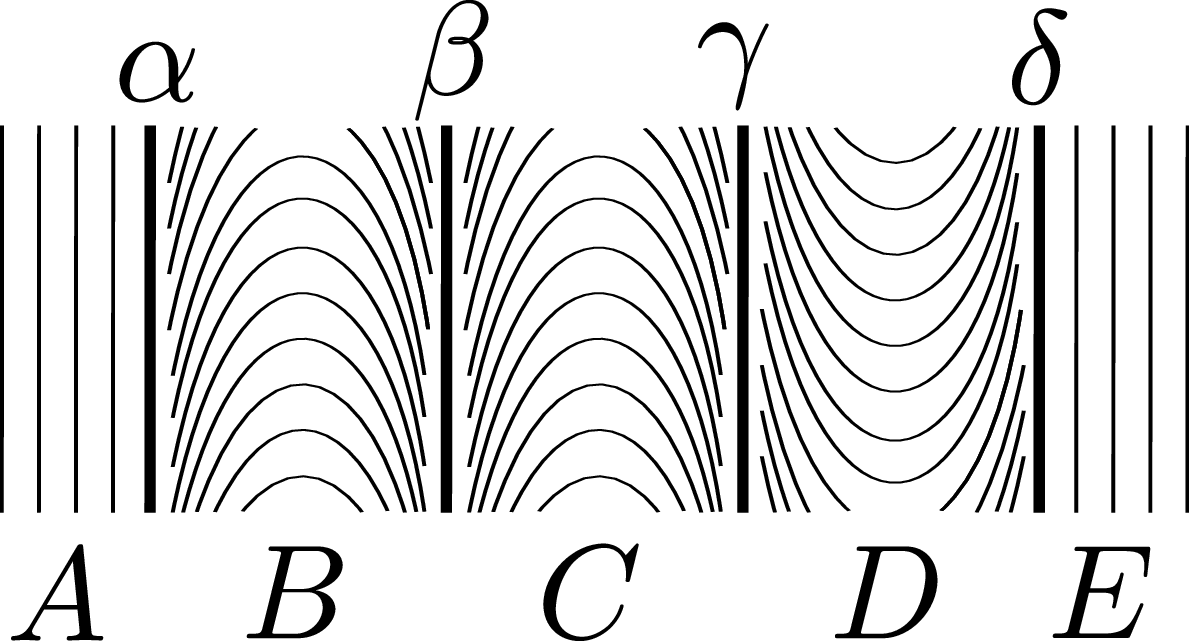} & \qquad &
\includegraphics[height=1.7cm]{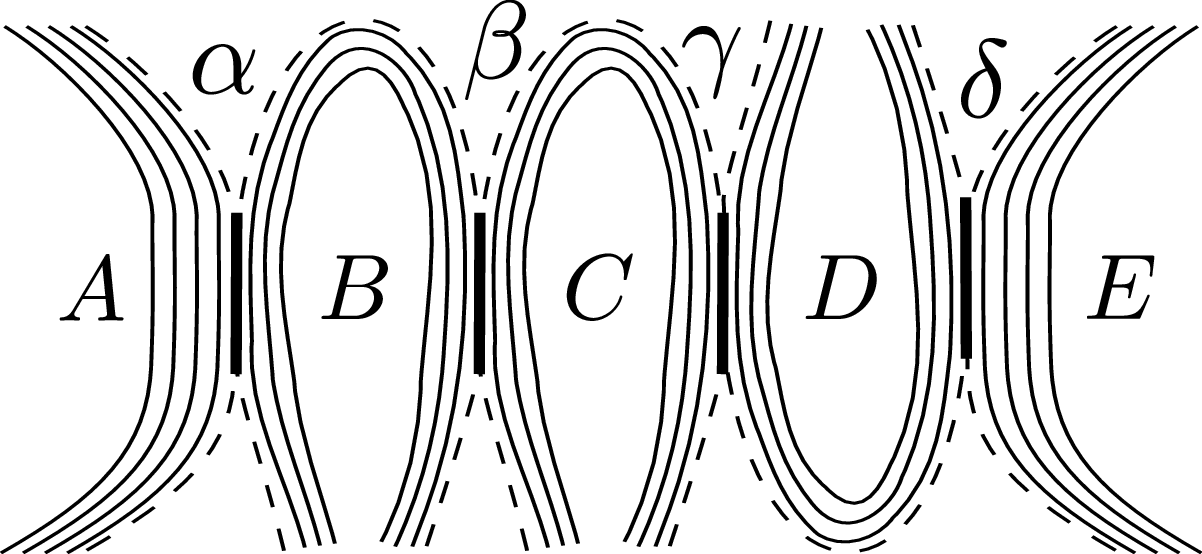} & \qquad &
\includegraphics[height=1.7cm]{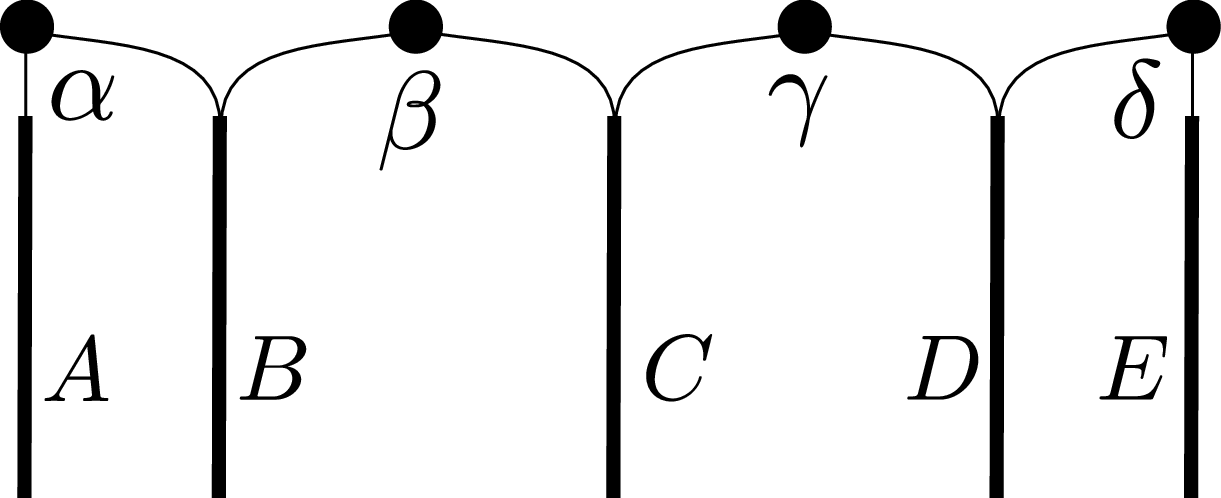} \\
(a) Foliation $\Partition$ & &
(b) Strips decomposition & &
(c) Space of leaves $\Ysp = \Xsp/\Partition$
\end{tabular}
\caption{}\label{fig:example:special_leafs}
\end{figure}

\begin{definition}
A subset $\strip \subset \bR^2$ will be called a \emph{model strip} if there exist $a<b$ such that
\begin{enumerate}
\item[\rm(1)] $\bR\times(a,b) \ \subset \ \strip \ \subset \ \bR\times[a,b]$;
\item[\rm(2)] the intersection $\strip \ \cap \ \bR\times\{a,b\}$ is a disjoint union of open intervals.
\end{enumerate}
Put
\begin{align*}
\partial_{-} \strip &= \strip \cap (\bR \times \lbrace a \rbrace), &
\partial_{+} \strip &= \strip \cap (\bR \times \lbrace b \rbrace), &
\partial S &= \partial_{-} S \cup \partial_{+} S.
\end{align*}
A model strip $\bR\times(a,b)$ will be called \emph{open}.
\end{definition}
Each model strip $\strip$ admits a natural $1$-dimensional foliation into parallel lines $\bR\times t$ and boundary intervals from $\partial\strip$.
We will call this foliation \emph{canonical}.
The following lemma implies that this foliation belongs to class $\classFol$.

\begin{lemma}\label{lm:leaf_shrinking}{\rm
(e.g.~\cite{MaksymenkoPolulyakh:PGC:2015}).}
Let $a<b \in \bR$, $\Xsp = \bR^2\setminus\bigl( (-\infty,a]\cup[b,+\infty) \bigr)$, and $\eps>0$.
Then there exists a homeomorphism $\phi:\bR^2\to\Xsp$ such that
\begin{enumerate}
\item[\rm(a)]
$\phi$ is fixed outside $\bR\times(-\eps,\eps)$;
\item[\rm(b)]
$\phi$ preserves foliations by horizontal lines, that is $\phi(\bR\times t) = t\times\bR$ for $t\not=0$ and $\phi(\bR\times 0) = (a,b)\times 0$, see Figure~\ref{fig:leaf_shrinking}.
\end{enumerate}
\end{lemma}

\begin{figure}[ht]
\includegraphics[height=1.5cm]{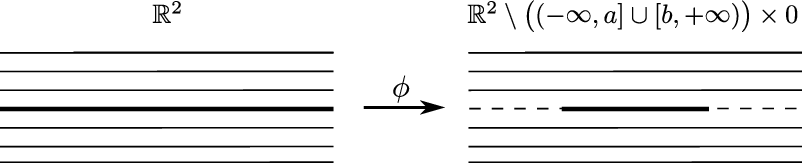}
\caption{}\label{fig:leaf_shrinking}
\end{figure}

\begin{example}
The foliation in Example~\ref{exmp:special_leaves} splits into five model strips such that
\begin{align*}
&A \ \cong \ E \ \cong \ \bR\times(0,1) \ \bigcup \ (0,1) \times 1, \\
& B \ \cong \ C \ \cong \ D \ \cong \ \bR\times(0,1) \ \bigcup \ \bigl( (0,1) \cup (2,3) \bigr) \times 1.
\end{align*}
\end{example}

Let $\bR\times[-1,1]$ be a model strip, $\phi_{+}, \phi_{-}:\bR\times\{-1\} \to \bR\times\{+1\}$ be two homeomorphisms given by
\begin{align*}
\phi_{+}(t,-1) &= (t,1),  &
\phi_{-}(t,-1) &= (-t,1),
\end{align*}
for $t\in\bR$, and $C = \bR\times[-1,1] / \phi_{+}$ and $M = \bR\times[-1,1] / \phi_{-}$ be the quotient spaces.
Thus $C$ (resp. $M$) is obtained from $\bR\times[-1,1]$ by identifying its boundary lines via preserving (resp. reversing) orientation homeomorphism.
Therefore $C$ is a cylinder and $M$ is a M\"obius band.
Moreover, the canonical foliation on $\bR\times[-1,1]$ yields certain foliations $\Partition_C$ and $\Partition_M$ on $C$ and $M$ respectively also belonging to class $\classFol$.
We will call $C$ a \emph{standard cylinder} and $M$ a \emph{standard M\"obius band}.

\subsection*{Foliation associated with a regular function.}
A continuous function $f:\bR^2\to\bR$ will be called \emph{regular} whenever for each $z\in\bR^2$
there are local coordinates $(u,v)$ in which $z=(0,0)$ and $f(u,v) = u+\mathrm{const}$.

It follows that the partition $\Partition$ of $\bR^2$ into connected components of level-sets $f^{-1}(t)$, $t\in\bR$,
of $f$ is a \emph{foliation} is a usual sense, i.e.\! it is \emph{locally} homeomorphic with a partition of $\bR^2$
into parallel lines.
We will say that $\Partition$ is a \emph{foliation associated with $f$}.

Notice that \emph{$f$ has no local extremes}, whence all leaves of $f$ are homeomorphic with $\bR$.
Indeed, if $\Partition$ has a closed leaf $\omega$, then by Jordan theorem $\omega$ bounds a $2$-disk.
Since $f$ is constant on $\omega$, it must have a local extreme inside that disk, which gives a contradiction.

Let $J\subset\bR$ be a connected subset, i.e.\! either open or closed or half-closed interval.
Then by a \emph{cross-section} $\crosssect:J\to\bR^2$ of $\Partition$ we will mean a continuous path intersecting each leaf at most once.
It easily follows that $\crosssect$ is a cross-section if and only if the composition $f\circ\crosssect:J\to\bR$ is strictly monotone.

By a \emph{saturation of a cross-section} $\crosssect:J\to\bR^2$ we will mean the saturation of its image $S(\crosssect(J))$ and denote it simply by $S(\crosssect)$, c.f.~\cite[\S1.4]{Kaplan:DJM:1940}

Kaplan~\cite[Theorem~30]{Kaplan:DJM:1940} proved that for a cross-section $\crosssect:[a,b]\to\bR^2$ of $\Partition$ its saturation $S(\crosssect)$ is foliated homeomorphic with $\bR\times[a,b]$ foliated by parallel lines.
However, this result can be misleading, since $S(\crosssect)$ is not necessarily a closed subset of $\Xsp$.

For instance, consider the foliation in Figure~\ref{fig:example:special_leafs}(b).
Let $\sigma:[a,b]\to\bR^2$ be a cross-section passing through the special leaf $\alpha$ and such that $\sigma(a)\in A$ and $\sigma(b) \in B$.
Then $\overline{S(\crosssect)} \setminus S(\crosssect) = \beta$.

\subsection*{Kaplan's construction.}
In~\cite[Theorem~29]{Kaplan:DJM:1940} W.~Kaplan has shown that the foliation $\Partition$ associated with a regular function $f$ belongs to class $\classFol$.
In fact, he associated to $\Partition$ a family of pairs $\xi = \{(\omega_i,\crosssect_i)\}_{i=-a}^{b}$ for some $a,b\in\bN\cup\{\infty\}$, where
\begin{itemize}
\item [(i)]
$\omega_i$ is a leaf being special for $i\not=0$;
\item[(ii)]
$\crosssect_0:(-1,1)\to\bR^2$, $\crosssect_i:[0,1)\to\bR^2$ for $i>0$, and $\crosssect_i:(-1,0]\to\bR^2$ for $i<0$ are certain proper cross-sections of $\Partition$;
\item[(iii)]
$\crosssect_{i}(0)  \in  \omega_i$ for all $i$,
\begin{align*}
\crosssect_i[0,1) \ \cap \ S \Bigl( \mathop{\cup}\limits_{j=0}^{i-1}\crosssect_j[0,1) \Bigr) &= \crosssect_{i}(0),   & i>0, \\
\crosssect_i(-1,0] \ \cap \ S \Bigl( \mathop{\cup}\limits_{j=i+1}^{0}\crosssect_j(-1,0] \Bigr) &= \crosssect_{i}(0), & i<0.
\end{align*}
\end{itemize}
Kaplan proved that $\xi$ determines $\Partition$ up to a foliated homeomorphism.

As noted above $S(\crosssect_{0})$ is foliated homeomorphic with $\bR\times(0,1)$ while $S(\crosssect_{i})$, $i\not=0$, is foliated homeomorphic with a strip $\bR\times[0,1)$.
Therefore the family $\xi$ determines at most countable family of strips $\{\Vsp_i = S(\crosssect_i)\}$ such that $\Vsp_{i+1}$ is glued to $\Vsp_i$ along the interval $\omega_i$ in their boundaries.

Kaplan's aim was to decrease the family of such strips as much as possible, see first paragraph of~\cite[Section~3.1]{Kaplan:DJM:1940}.
However, the construction of family $\xi$ then becomes ambiguous and depends on a particular choice of special leaves and cross-sections.
This is illustrated in Figure~\ref{fig:example:Kaplan_construction}(b), where two such families for the same foliation
are presented.

\begin{figure}[ht]
\begin{tabular}{ccc}
\includegraphics[height=3cm]{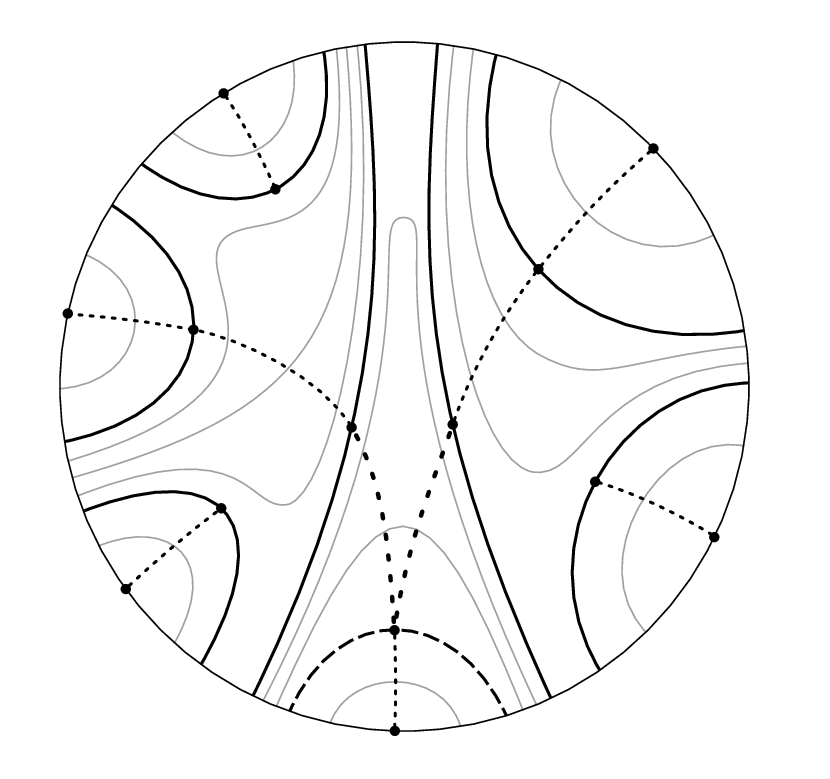} & \qquad\qquad &
\includegraphics[height=3.3cm]{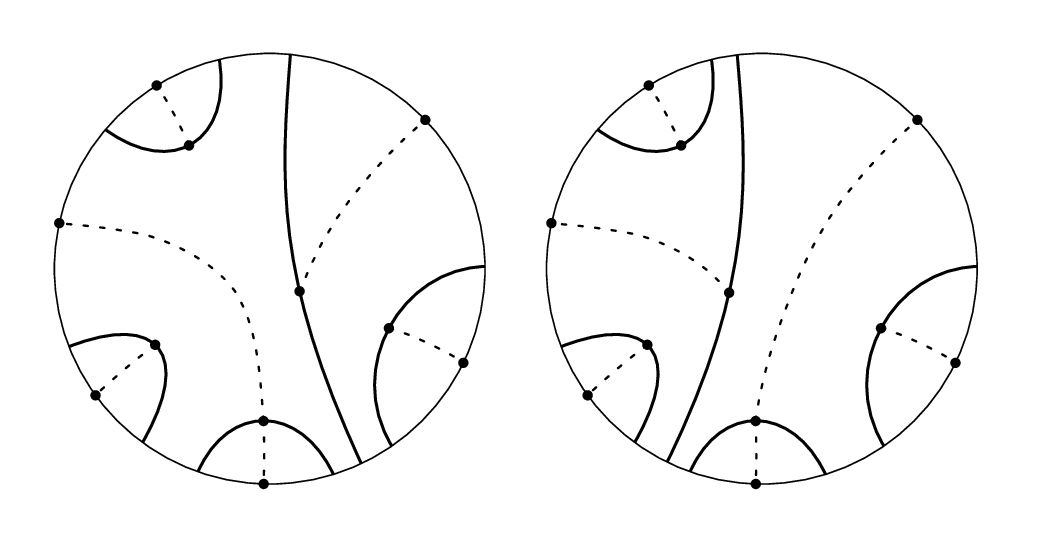} \\
(a) Foliation & &
(b) Two distinct maximal families of cross-sections
\end{tabular}
\caption{}\label{fig:example:Kaplan_construction}
\end{figure}

On the other hand cutting $\bR^2$ along special leaves is an unambiguous procedure and it gives a canonical decomposition of $\bR^2$.

In the present paper we extends Kaplan's results to foliations $\Partition$ from class $\classFol$ on arbitrary surfaces $\Xsp$ and describe the topological structure of connected components of $\Xsp\setminus\Sigma$ and their closures, where $\Sigma$ is the union of all special leaves of $\Partition$.

\begin{theorem}\label{th:open_strips}
Let $\Xsp$ be a connected $2$-dimensional manifold and $\Partition$ be a foliation on $\Xsp$ belonging to class $\classFol$.
Suppose that the family $\Sigma$ of all special leaves of $\Partition$ is locally finite, and let $\Qsp$ be a connected component of $\Xsp\setminus(\Sigma \cup \partial\Xsp)$.
Then the following statements hold true.
\begin{enumerate}
\item\label{th:open_strips:Q}
$\Qsp$ is foliated homeomorphic either with a standard cylinder $C$ or a standard M\"obius band $M$ or an open model strip $\bR\times(-1,1)$.
Moreover, in the first two cases $\Qsp = \Xsp$.

\item\label{th:open_strips:closure_of_Q}
Suppose $\Qsp$ is foliated homeomorphic with an open model strip.
Fix any foliated homeomorphism $\phi:\bR\times(-1,1)\to\Qsp$ and denote
\begin{align*}
\Qmin &= \phi\bigl(\bR\times(-1,0]\bigr), & \Qmax &= \phi\bigl(\bR\times[0,1)\bigr).
\end{align*}
Then the closures $\overline{\Qmin}$ and $\overline{\Qmax}$ are foliated homeomorphic to some model strips.
\end{enumerate}
\end{theorem}

This theorem implies that the topological structure of the foliation $\Partition\in\classFol$ is uniquely determined by the combinatorics of gluing model strips.
Also notice that the intersection $(\overline{\Qmin}\cap\overline{\Qmax}) \setminus \phi(\bR\times0)$ can be non-empty, whence one can not expect that $\overline{\Qsp}=\overline{\Qmin}\cup\overline{\Qmax}$ is homeomorphic with a model strip.

The proof of Theorem~\ref{th:open_strips} will be given in \S\ref{sect:proof:1:th:open_strips} and~\ref{sect:proof:2:th:open_strips}.

\section{Special points of non-Hausdorff spaces}\label{sect:spec_points}

Throughout this section $\Ysp$ be a topological space.

\begin{definition}\label{def:spec_point}
Let $y\in \Ysp$ and $\beta_{y}$ be the family of all neighborhoods of $y$.
Then the following set
\[
\bnd{y} := \mathop{\cap}\limits_{\Vsp\in\beta_{y}} \overline{\Vsp}
\]
will be called the \emph{Hausdorff closure} of $y$.
We will say that $y$ is a \emph{special} point of $\Ysp$ whenever $y\not=\bnd{y}$.
The set of all special points of $\Ysp$ will be denoted by $\YspecPtSet$.
\end{definition}

Notice that $\Ysp$ is Hausdorff if and only if $y=\bnd{y}$ for all $y\in\Ysp$, i.e.\! when $\YspecPtSet=\varnothing$.

\begin{lemma}\label{lm:spec_points}
\begin{enumerate}
\item\label{enum:lm:prop:symm_pt_bnd1}
Let $y,z\in\Ysp$.
Then $y \in \bnd{z}$ if and only if $z\in \bnd{y}$, however, in general, $\bnd{y} \not=\bnd{z}$.

\item\label{enum:lm:prop:img_of_bnd_pt1}
Let $f:\Ysp \to \Zsp$ be a continuous map into a Hausdorff topological space $\Zsp$.
Then $f(\bnd{y}) = f(y)$ for all $y\in\Ysp$.

\item\label{enum:lm:prop:compl_to_spec_Hausdorff1}
The set $\Ysp\setminus\YspecPtSet$ of all non-special points is Hausdorff.
\end{enumerate}
\end{lemma}

\begin{proof}
(\ref{enum:lm:prop:symm_pt_bnd1})
Suppose $y\in\bnd{z} = \mathop{\cap}\limits_{\Vsp\in\beta_{z}} \overline{\Vsp}$, that is $y$ belongs to the closure of each neighborhood of $z$ which means in turn that every neighborhood of $y$ intersect every neighborhood of $z$.
The latter property is symmetric with respect to $y$ and $z$, whence $z\in\bnd{y}$ as well.

\smallskip

(\ref{enum:lm:prop:img_of_bnd_pt1})
Suppose $z\in \bnd{y}$ but $f(y) \not= f(z)$.
Since $\Zsp$ is Hausdorff, there exist open disjoint neighborhoods $\Wsp_{f(y)}$ and $\Wsp_{f(z)}$ of points $f(y)$ and $f(z)$.
But then their inverses $\Vsp_y = f^{-1}(\Wsp_{f(y)})$ and $\Vsp_z = f^{-1}(\Wsp_{f(z)})$ are disjoint open neighborhoods of $y$ and $z$ respectively.
Hence $z$ can not belongs to $\bnd{y}$ which contradicts to the assumption.

\smallskip

(\ref{enum:lm:prop:compl_to_spec_Hausdorff1})
Let $y,z\in\Ysp\setminus\YspecPtSet$ be two distinct points.
Thus $y=\bnd{y} \not =z$ and so there exist disjoint neighborhoods $\Vsp_y$ and $\Vsp_z$ of $y$ and $z$ respectively.
This implies that $\Ysp\setminus\YspecPtSet$ is Hausdorff.
\end{proof}

\subsection*{Non-Hausdorff one-dimensional manifolds}
Let $\Ysp$ be a $T_1$-topological space locally homeomorphic with open sets of $[0,1)$.
Notice that we allow $\Ysp$ to be non-Hausdorff.
Then as usual the set of points having an open neighborhood homeomorphic with $(0,1)$ will be denoted by $\Int{\Ysp}$ and called the \emph{interior} of $\Ysp$, while its complement $\partial\Ysp:=\Ysp\setminus\Int{\Ysp}$ will be called the \emph{boundary} of $\Ysp$.

\begin{lemma}\label{lm:connected_components_of_nonspec_pts}
Suppose that the set $\YspecPtSet$ of special points of $\Ysp$ is locally finite.
Then every connected component $\Wsp$ of $\Ysp\setminus\YspecPtSet$ is open in $\Ysp$ and is homeomorphic with one of the following spaces: $[0,1)$, $(0,1)$, $[0,1]$, $S^1$.
In the last two cases, i.e.\! when $\Wsp$ is compact, $\Wsp$ is a connected component of $\Ysp$.

Every connected component of $\Ysp\setminus(\YspecPtSet\cup\partial\Ysp)$ is homeomorphic with $(0,1)$.
\end{lemma}
\begin{proof}
Since $\Ysp$ is a $T_1$-space, every point $y\in\Ysp$ is a closed subset.
Also since $\YspecPtSet$ is locally finite, it follows that $\YspecPtSet$ is a closed subset,
whence by~(\ref{enum:lm:prop:compl_to_spec_Hausdorff1})
of Lemma~\ref{lm:spec_points} $\Ysp\setminus\YspecPtSet$ is a Hausdorff topological space locally homeomorphic with $[0,1)$.
Hence every connected component $\Wsp$ of $\Ysp\setminus\YspecPtSet$ is a one-dimensional manifold and so it is homeomorphic with one of the spaces $[0,1)$, $(0,1)$, $[0,1]$, $S^1$.
Moreover, since $\Ysp\setminus\YspecPtSet$ is locally connected, we obtain that $\Wsp$ is open in $\Ysp\setminus\YspecPtSet$ and therefore in $\Ysp$ as well.

Suppose $\Wsp$ is compact, i.e.\! it is homeomorphic either with $[0,1]$ or with $S^1$.
Let us show that then $\Wsp$ is also closed in $\Ysp$.
This will imply that $\Wsp$ is a connected component of $\Ysp$.

Let $\{y_i\}_{i\in\bN} \subset \Wsp$ be a sequence converging to some $z\in\Ysp$.
We should prove that $z\in\Wsp$.
Since $\Wsp$ is compact, that sequence also converges to some $y\in\Wsp$.
Hence if $\Vsp_y$ and $\Vsp_z$ are any two open neighborhoods of $y$ and $z$ respectively, then there exists $n>0$ such that $y_n\in\Vsp_y\cap\Vsp_z$.
Thus $\Vsp_y\cap\Vsp_z\not=\varnothing$, which implies that $z\in\bnd{y} = \{y\}$, that is $z=y\in\Wsp$.

We leave the last statement for the reader.
\end{proof}

Suppose $\Ysp$ is connected and not homeomorphic with a circle.
Let $\{ \Wsp_{\ai} \}_{\ai\in\aSet}$ be the family of all connected components of $\YspecPtSet \cup \partial\Ysp$.
Then due to Lemma~\ref{lm:connected_components_of_nonspec_pts} for each $\ai\in\aSet$ there exists a homeomorphism $\phi_{\ai}:(-1,1) \to \Wsp_{\ai}$.
Consider the following collection of subsets:
\begin{align*}
\bSet &= \bigl\{ \phi_{\ai}(-1,-\onehalf], \ \phi_{\ai}[\onehalf, 1) \bigr\}_{\ai\in\aSet}\,.
\end{align*}
Let $\Jint\in\bSet$.
Then we denote $\JOint := \phi_{\ai}(-1,-\onehalf)$ if $\Jint= \phi_{\ai}(-1,-\onehalf]$, and $\JOint:=\phi_{\ai}(\onehalf,1)$ if $\Jint= \phi_{\ai}[\onehalf,1)$ for some $\ai\in\aSet$.
Thus each $\Jint\in\bSet$ is homeomorphic with a half-open segment $[0,1)$, and $\JOint$ is the subset of $\Jint$ corresponding to $(0,1)$.

\begin{lemma}\label{lm:nbh_spec_points}
Let $y\in\partial\Ysp$.
Then there exists a unique $\Jint\in\bSet$ such that $y\in\overline{\Jint}$.
In this case $\Vsp:=\{y\} \cup \JOint$ is an open neighborhood of $y$ and there exists a homeomorphism $\psi:[0,1) \to \{y\} \cup \Jint$ such that $\psi(0)=y$ and $\psi(0,1)=\Vsp$.

Suppose $y\in\YspecPtSet\setminus\partial\Ysp$.
Then there exist two distinct elements $\JA, \JB\in\bSet$ such that $y\in\overline{\JA}\cap\overline{\JB}$ and $y\not\in\overline{\JC}$ for all other $\JC\in\bSet$.
Moreover, the set $\Vsp:=\JOA \cup \{y\} \cup \JOB$ is an open neighborhood of $y$ and there exists a homeomorphism $\mu:[-1,1] \to \JA \cup \{y\} \cup \JB$ such that $\psi(0)=y$ and $\psi(-1,1)=\Vsp$.
\end{lemma}
\begin{proof}
We will consider only the case $y\in\YspecPtSet\setminus\partial\Ysp$.
Notice that the family $\{\phi_{\ai}[-1,1]\}_{\ai\in\aSet}$ is locally finite and consists of closed sets.
Therefore its union $\Zsp = \mathop{\cup}\limits_{\ai\in\aSet} \phi_{\ai}[-1,1]$ is closed.
Hence the set $T = (\YspecPtSet\setminus\{y\}) \cup \partial\Ysp \cup \Zsp$ is closed and does not contain $y$.
Therefore there exits a neighborhood $\Jj\subset \Ysp\setminus T$ of $y$ and a homeomorphism $\mu: [-\onehalf,\onehalf] \to \Jj$ such that $\mu(0)=y$ and $\mu(-\onehalf,\onehalf)$ is an open neighborhood of $y$.

Notice that $\Jj\setminus\{y\}$ consists of exactly two connected components $\Ac=\mu[-\onehalf,0)$ and $\Bc=\mu(0,\onehalf]$, and is contained in $\Ysp\setminus \bigl( \YspecPtSet \cup \partial\Ysp \cup \Zsp\bigr) = \mathop{\cup}\limits_{\Jint\in\bSet}\JOint$.
Hence $\Ac \subset \JOA$ and $\Bc \subset \JOB$ for some $\JA, \JB\in\bSet$, see Figure~\ref{fig:spec_pt_nbh}.

\begin{figure}[ht]
\includegraphics[height=2.3cm]{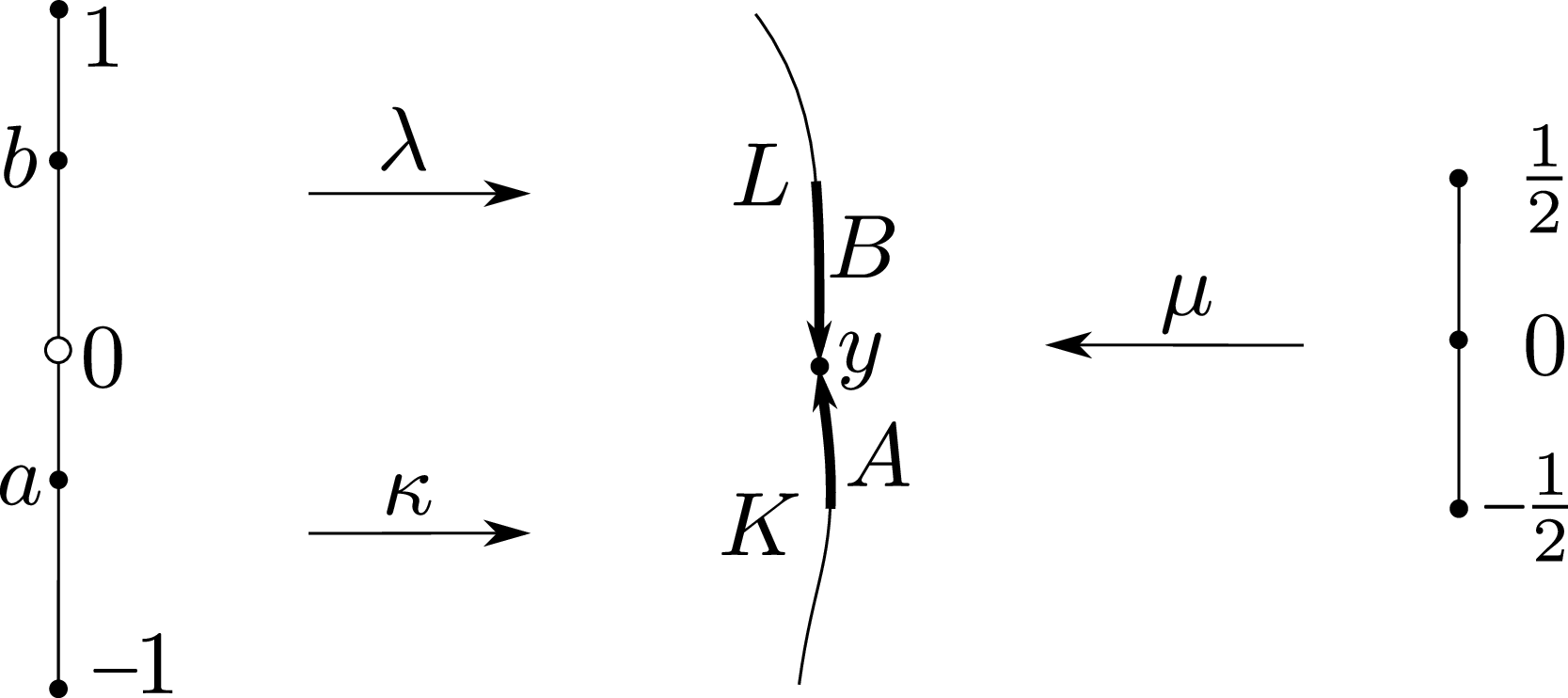}
\caption{}\label{fig:spec_pt_nbh}
\end{figure}

Moreover, any other neighborhood of $y$ intersects both $\Ac$ and $\Bc$ and therefore both $\JA$ and $\JB$.
Hence $y\in\overline{\JA}\cap\overline{\JB}$ and $y\not\in\overline{\JC}$ for all other $\JC\in\bSet$ distinct from $\JA$ and $\JB$.

Fix any homeomorphisms $\ahom:[-1,0) \to \JA$ and $\bhom:(0,1] \to \JB$.
Notice that $\Ac$ is not contained in any compact subset $P$ of $\JA$, since otherwise $y \in\overline{\Ac} \subset P \subset \JA$, which contradicts to the assumption that $y\not\in\JA$.
This implies that $\ahom^{-1}(\Ac) = [a,0) \subset (-1,0)$, where $a=\ahom^{-1}\circ\mu(-\onehalf)\in(-1,0)$.
By the same arguments, $\bhom^{-1}(\Bc) = (0,b]\subset (0,1)$, where $b=\bhom^{-1}\circ\mu(\onehalf)\in(0,1)$.

\begin{sublemma}
$\JA \not=\JB$.
\end{sublemma}
\begin{proof}
If $\JA=\JB$, then we have a homeomorphism $\chom=\bhom^{-1}\circ\ahom:[-1,0)\to(0,1]$.
Hence there exists $c \in(a,0)$ such that $c'=\chom(c) \in (0,b)$.
Then $\ahom(c) = \Ac$ and $\bhom\circ\chom(c)\in\Bc$.
But $\bhom\circ\chom(c) = \ahom(c)$, and so $\Ac\cap\Bc\not=\varnothing$, which contradicts to the assumption.
\end{proof}

Now fix arbitrary orientation preserving homeomorphisms $\aahom:[-1,-\onehalf]\to[-1,a]$ and
$\bbhom:[\onehalf,1]\to[b,1]$ and define the map $\psi:[-1,1] \to \JA \cup \{y\} \cup \JB$ by the
formula
\[
\psi(t) =
\begin{cases}
\ahom^{-1}\circ \aahom(t), & t\in[-1, -\onehalf], \\
\mu(t), & t\in[-\onehalf, \onehalf],\\
\bhom^{-1}\circ \bbhom(t), & t\in[\onehalf,1].
\end{cases}
\]
One easily checks that $\psi$ is a required homeomorphism.
\end{proof}

\section{Partitions}
Let $\Xsp$ be a topological space, $\Partition$ be a partition of $\Xsp$, $\Ysp = \Xsp/\Partition$ be the quotient space, and $\prj:\Xsp\to\Ysp$ be the corresponding quotient map.
We will endow $\Ysp$ with the \emph{factor topology}, so a subset $\Vsp \subset \Ysp$ is open if and only if its inverse $\prj^{-1}(\Vsp)$ is open in $\Xsp$.

A \emph{saturation} $S(\Usp)$ of a subset $\Usp\subset X$ with respect to $\Partition$ is the union of all $\elem \in\Partition$ such that $\elem\cap \Usp\not=\varnothing$.
Equivalently, $S(\Usp) = \prj^{-1}(\prj(\Usp))$.
A subset $\Usp$ is \emph{saturated} if $\Usp = S(\Usp)$.
Evidently, if $A \cap S(B) = \varnothing$, then $S(A) \cap S(B)=\varnothing$ as well.

\begin{lemma}\label{lm:prop}
\begin{enumerate}
\item\label{enum:lm:prop:T1_space}
$\Ysp$ is a $T_1$-space if and only if each element $\elem\in\Partition$ is closed.

\item\label{enum:lm:prop:prj_open_equiv_conditions}
The following conditions are equivalent:
\begin{enumerate}
\item\label{enum:lm:prop:e:prj_open}
the map $\prj:\Xsp\to\Ysp$ is open;
\item\label{enum:lm:prop:e:saturation_is_open}
for each $x\in\Xsp$ there exists an open neighborhood $\Usp$ whose saturation $S(\Usp)$ is open;
\item\label{enum:lm:prop:e:cover_with_open_restr}
there exists an open cover $\beta=\{\Usp_i\}_{i\in\Lambda}$ of $\Xsp$ such that for each $i\in\Lambda$ the restriction $\prj|_{\Usp_i}: \Usp_i\to\prj(\Usp_i)$ is an open map.
\end{enumerate}

\item\label{enum:lm:prop:prj_open}
If $\prj$ is open then for each saturated subset $B$ we have that
\begin{gather}
\label{equ:X_setminus_ovrSA}
\Xsp\setminus\overline{B} = S(\Xsp\setminus\overline{B}), \\
\label{equ:image_of_closure_B}
p(\overline{B}) = \overline{p(B)}.
\end{gather}
In particular, $\overline{S(A)}$ and $\Xsp\setminus\overline{S(A)}$ are saturated for each subset $A \subset \Xsp$.

\item\label{enum:lm:prop:loc_finite_families}
Let $\beta = \{\Wsp_i\}_{i\in\Lambda}$ be a family of subsets of $\Ysp$, and $\alpha = \{\prj^{-1}(\Wsp_i)\}_{i\in\Lambda}$ be the corresponding family of their inverses in $\Xsp$.
If $\beta$ is locally finite, then so is $\alpha$.
Conversely, if $\alpha$ is locally finite and $\prj$ is open then $\beta$ is locally finite as well.

\item\label{enum:lm:prop:loc_finite_families_nbh}
Suppose $\Xsp$ is a normal topological space and $\alpha = \{\elem_i\}_{i\in\bN}$ is a locally finite family of mutually disjoint closed subsets of $\Xsp$.
Then for each $i\in\bN$ there exists a neighborhood $\Usp_i$ of $\elem_i$ such that $\overline{\Usp_i}\cap\overline{\Usp_j}=\varnothing$ for $i\not=j$.

\item\label{enum:lm:prop:homeomorphism}
Let $f:A \to B$ be a bijection between topological spaces.
Suppose that $\{\Ksp_i\}_{i\in\Lambda}$ is a locally finite cover of $A$ by closed sets.
If each of the restrictions $f|_{\Ksp_i}: \Ksp_i\to B$ is continuous, then $f$ is continuous it self.

Moreover, suppose the family $\{\psi(\Ksp_i)\}_{i\in\Lambda}$ is locally finite, $f(\Ksp_i)$ is closed in $B$, and the restriction $f|_{\Ksp_i}: \Ksp_i\to f(\Ksp_i)$ is a homeomorphism for each $i\in\Lambda$.
Then $f$ is a homeomorphism.
\end{enumerate}
\end{lemma}

\begin{proof}
Statements (\ref{enum:lm:prop:T1_space}), (\ref{enum:lm:prop:prj_open_equiv_conditions}),
and~(\ref{enum:lm:prop:homeomorphism}) are easy and we leave them for the reader.

\smallskip

(\ref{enum:lm:prop:prj_open})
Suppose $\prj$ is an open map and let $B \subset \Xsp$ be a saturated subset.
Then $\Xsp\setminus B$ is also saturated, i.e.\! $S(\Xsp\setminus B)=\Xsp\setminus B$, and so
\[
\Xsp\setminus\overline{B} \ \subset \ S(\Xsp\setminus\overline{B}) \ \subset \ S(\Xsp\setminus B) = \Xsp\setminus B.
\]
Hence
\[
\overline{B} \ \supset \ \Xsp\setminus S(\Xsp\setminus\overline{B}) \ \supset \ B.
\]
As $\Xsp\setminus\overline{B}$ is open, $S(\Xsp\setminus\overline{B})$ is open as well, and therefore $\Xsp\setminus S(\Xsp\setminus\overline{B})$ is a closed subset containing $B$.
Therefore it must contain the closure $\overline{B}$, hence $\overline{B} = \Xsp\setminus S(\Xsp\setminus\overline{B})$, which implies~\eqref{equ:X_setminus_ovrSA}.

\smallskip

Let us prove~\eqref{equ:image_of_closure_B}.
Since $p$ is continuous, $p^{-1}(\overline{p(B)})$ is a closed subset containing $B$.
Therefore it contains $\overline{B}$, and so $p(\overline{B}) \subset \overline{p(B)}$.

Conversely, by~\eqref{equ:X_setminus_ovrSA}, $\overline{B}$ is saturated and closed.
Therefore, by definition of the quotient topology, $p(\overline{B})$ is a closed subset and it contains $p(B)$.
Hence it also contains $\overline{p(B)}$, i.e.\! $p(\overline{B}) \supset \overline{p(B)}$.

\smallskip

(\ref{enum:lm:prop:loc_finite_families})
Suppose $\beta$ is a locally finite family and $x\in \Xsp$.
We should find a neighborhood $\Usp$ of $x$ which intersects only finitely many elements from $\alpha$.
Let $y=\prj(x)$.
Since $\beta$ is locally finite, there exists a neighborhood $\Vsp$ of $y$ intersecting only finitely many elements $\Wsp_{i_1},\ldots,\Wsp_{i_k} \in \beta$.
Then $\prj^{-1}(\Vsp)$ is an open neighborhood of $x$ intersecting only the following elements $\prj^{-1}(\Wsp_{i_1}),\ldots,\prj^{-1}(\Wsp_{i_k})$ of $\alpha$.

Conversely, suppose $\alpha$ is locally finite and $\prj$ is open.
Let $y\in\Ysp$ and $x\in\Xsp$ be such that $\prj(x)=y$.
Then there exists a neighborhood $\Usp$ of $x$ intersecting only finitely many elements, say $\prj^{-1}(\Wsp_{i_1}),\ldots,\prj^{-1}(\Wsp_{i_k})$, of $\alpha$.
Therefore its saturation $S(\Usp) = \prj^{-1}(\prj(\Usp))$ also intersects only $\prj^{-1}(\Wsp_{i_1}),\ldots,\prj^{-1}(\Wsp_{i_k})$.

Since $\prj$ is open, the image $\prj(\Usp)$ is an open neighborhood of $y$.
We claim that $\prj(\Usp)$ intersects only the elements $\Wsp_{i_1},\ldots,\Wsp_{i_k}$ of $\alpha$.
Indeed, if $\prj(\Usp) \cap \Wsp_i \not=\varnothing$ for some $i\in \Lambda$, then $\prj^{-1}(\prj(\Usp)) \cap \prj^{-1}(\Wsp_i)\not=\varnothing$ which is possible only when $i \in \{i_1,\ldots,i_k\}$.

\smallskip

(\ref{enum:lm:prop:loc_finite_families_nbh})
For each $i\in\bN$ consider the following subfamily $\alpha_i = \{ \omega_j \}_{j\geq i}$ of $\alpha$, so $\alpha=\alpha_1$ and $\alpha_{i+1}  \subset \alpha_i$ for all $i\in\bN$.
Then each $\alpha_{i}$ is locally finite as well, and therefore the union $A_i = \mathop{\cup}\limits_{j=i}^{\infty} \elem_i$ is a closed subset of $\Xsp$.

Since $\Xsp$ is normal and $\elem_1$ and $A_2$ are mutually disjoint and closed, there exists an open neighborhood $\Usp_1$ of $\elem_1$ such that $\overline{\Usp_1} \cap A_2 = \varnothing$.
Then $\elem_2$ and $\overline{\Usp_1} \cup A_3$ are mutually disjoint and closed, whence there exists an open neighborhood $\Usp_2$ of $\elem_2$ which does not intersect $\overline{\Usp_1} \cup A_3$.
Repeating these arguments so on we will construct for each $i\in\bN$ an open neighborhood $\Usp_i$ of $\elem_i$ such that $\overline{\Usp_i}$ does not intersect $\bigl(\cup_{j=1}^{i-1} \overline{\Usp_j}\bigr)\cup A_{i+1}$.
Then $\overline{U_i} \cap \overline{U_j} = \varnothing$ for all $i\not=j\in\bN$.
\end{proof}

\begin{definition}\label{def:loc_triv_partition}
We will say that a partition $\Partition$ is \emph{locally trivial} if for each $\omega\in\Partition$ there exists an open neighborhood $\Usp$, a topological space $J$, a point $t_0\in J$, and a homeomorphism $\phi:\omega\times J \to \Usp$ such that $\phi(\omega \times t)$ is an element of $\Partition$ for all $t\in J$ and $\phi(x,t_0) = x$ for all $x\in\omega$.
\end{definition}
In particular, a foliation belonging to class $\classFol$ is a locally trivial partition.

Notice that in the notation of Definition~\ref{def:loc_triv_partition} $\Usp$ is saturated and open in $\Xsp$,
whence its image $\Vsp =\prj(\Usp)$ is open in $\Ysp$ and we have the following commutative diagram:

\begin{equation}\label{equ:reformulation_loc_triv_partition}
\xymatrix{
\omega\times J \ar[rr]^-{\phi}_-{\cong} \ar[d]_-{q_2} && \Usp = \prj^{-1}(\Vsp) \ar[d]^-{\prj} \\
J \ar[rr]^-{\xi}_-{1-1} && \Vsp
}
\end{equation}
where $q_2$ is a projection onto the second multiple and $\xi$ is the induced one-to-one continuous map but it is not necessarily a homeomorphism.

\begin{lemma}\label{lm:loc_triv_relations}
The following conditions are equivalent:
\begin{enumerate}
\item\label{lm:loc_triv_relations:enum:prj_is_ltfibr}
the quotient map $\prj:\Xsp\to\Ysp$ is a locally trivial fibration;
\item\label{lm:loc_triv_relations:enum:partit_lt__prj_open}
partition $\Partition$ is locally trivial and the quotient map $\prj:\Xsp\to\Ysp$ is open.
\end{enumerate}
\end{lemma}

\begin{proof}
\eqref{lm:loc_triv_relations:enum:prj_is_ltfibr}$\Rightarrow$\eqref{lm:loc_triv_relations:enum:partit_lt__prj_open}.
Suppose $\prj$ is a locally trivial fibration.
We claim that then $\Partition$ is locally trivial.
Indeed, let $\omega\in\Partition$ and $y=\prj(\omega)\in\Ysp$.
Since $\prj$ is locally trivial, there exists a neighborhood $\Vsp$ of $y$ and the following commutative
diagram:

\begin{equation}\label{equ:reformulation_loc_triv_fibration}
\xymatrix{
\omega\times \Vsp \ar[rr]^-{\phi}_-{\cong} \ar[d]_-{q_2} && \Usp_{\omega} = \prj^{-1}(\Vsp) \ar[d]^-{\prj} \\
 \Vsp \ar[rr]^-{\xi=\mathrm{id}_{\Vsp}} && \Vsp
}
\end{equation}
in which $\phi$ is a homeomorphism.
This diagram coincides with~\eqref{equ:reformulation_loc_triv_partition} for $J=\Vsp$, and therefore $\Partition$ is a locally trivial partition.

Let us prove that $\prj$ is an open map.
Notice that in Diagram~\eqref{equ:reformulation_loc_triv_fibration} $q_2$ is an open map as a coordinate projection.
Since $\phi$ is a homeomorphism, it follows that the restriction $\prj|_{\Usp_{\omega}}$ is an open map as well.
But then $\beta=\{\Usp_{\omega}\}_{\omega\in\Partition}$ is an open cover of $\Xsp$ such that each restriction $\prj|_{\Usp_{\omega}}$ is open.
Therefore by~(\ref{enum:lm:prop:prj_open_equiv_conditions}) of Lemma~\ref{lm:prop} $\prj$ is
open.

\eqref{lm:loc_triv_relations:enum:partit_lt__prj_open}$\Rightarrow$\eqref{lm:loc_triv_relations:enum:prj_is_ltfibr}.
Suppose $\prj$ is an open map and $\Partition$ is locally trivial.
We claim that then in~\eqref{equ:reformulation_loc_triv_partition} the map $\xi$ is open, and therefore it is a homeomorphism.
This will imply that $\prj$ is a locally trivial fibration.

Let $T \subset J$ be an open subset.
Then $\phi\circ q_2^{-1}(T)$ is open in $\Usp$.
Since $\prj$ is open, we get that $\xi(T) = \prj\circ\phi\circ q_2^{-1}(T)$ is open in $\Vsp$.
Thus $\xi$ is an open map.
\end{proof}

\begin{definition}\label{def:spec_element}
An element $\omega \in \Partition$ will be called \emph{special} if its image $y=\prj(\omega)\in\Ysp$ is a special point of $\Ysp$, i.e. $y\not=\bnd{y}:=\mathop{\cap}\limits_{\Vsp\in\beta_{y}} \overline{\Vsp}$, where $\beta_{y}$ is the family of all neighborhoods of $y$, see Definition~\ref{def:spec_point}.
Let also
\begin{align*}
\bnd{\omega} &= \underset{N(\omega)} {\bigcap} \overline{S\left( N(\omega) \right)}, &
\bnds{\omega} &= \underset{N_{S}(\omega)} {\bigcap} \overline{N_{S}(\omega)},
\end{align*}
where $N(\omega)$ runs over \emph{all open} neighborhoods of $\omega$ and $N_S(\omega)$ runs over \emph{all saturated open} neighborhoods of $\omega$.
\end{definition}

\begin{lemma}\label{lm:special_elements}
Let $\omega\in\Partition$ and $y = \prj(\omega)$.
Then
\begin{align*}
\bnd{\omega} & \ \subset \ \bnds{\omega} \ \subset \ \prj^{-1}(\bnd{y}).
\end{align*}
If $\prj$ is an open map, then
\begin{align*}
\bnd{\omega} &= \bnds{\omega} = \prj^{-1}(\bnd{y}), &
\prj(\bnds{\omega}) &= \bnd{y}.
\end{align*}
\end{lemma}

\begin{proof}
First we establish relations between $\bnd{\omega}$ and $\bnds{\omega}$.
Notice that the family $\mathcal{A} = \{ S(N(\omega)) \}$ of saturations of all open neighborhoods of $\omega$ includes the family $\mathcal{B} = \{ N_S(\omega) \}$ of all saturated open neighborhoods of $\omega$.
Therefore the intersection $\bnd{\omega}$ of the larger family $\mathcal{A}$ is contained in the intersection $\bnds{\omega}$ of the smaller family $\mathcal{B}$, that is $\bnd{\omega} \ \subset \ \bnds{\omega}$.

If $\prj$ is an open map, so the saturation of an open set is open, then $\mathcal{A} = \mathcal{B}$, and therefore $\bnd{\omega}=\bnds{\omega}$.

\smallskip

Now we will describe relationships between $\bnds{\omega}$ and $\bnd{y}$.
By definition of the quotient topology on $\Ysp$ the map $\prj$ induces a bijection between the families $\mathcal{B}$ and $\beta_{y}$.
Moreover, if $N_S(\omega) \in \mathcal{B}$ is an open saturated neighborhood of $\omega$ and $\Vsp = \prj(N_S(\omega))$ is an open neighborhood of $y$, then, due to continuity of $\prj$, we have that $\prj(\overline{N_S(\omega)}) \subset \overline{\Vsp}$.
Hence $\prj(\bnds{\omega}) \subset \bnd{y}$, that is $\bnds{\omega} \ \subset \ \prj^{-1}(\bnd{y})$.

If $\prj$ is open, then, due to~\eqref{equ:image_of_closure_B}, $\prj(\overline{N_S(\omega)}) = \overline{\prj(N_S(\omega))} = \overline{\Vsp}$, whence
\begin{equation}\label{equ:p_bnd_omega}
\prj(\bnds{\omega}) =
\prj\Bigl(\,\bigcap_{N_{S}(\omega)\in\mathcal{B}}\overline{N_{S}(\omega)} \,\Bigr) =
\prj\Bigl(\,\bigcap_{V \in \beta_{y}} \prj^{-1}(\overline{V}) \,\Bigr) =
 \bigcap_{V \in \beta_{y}} \overline{V} = \bnd{y}.
\end{equation}
Finally, as $\bnds{\omega}$ is saturated as an intersection of saturated sets, it follows from~\eqref{equ:p_bnd_omega} that $\bnds{\omega} = \prj^{-1}(\bnds{y})$.
\end{proof}

\begin{lemma}\label{lm:local_prop}
Suppose that the following conditions hold true:
\begin{enumerate}
\item[\rm(a)]
$\prj:\Xsp\to\Ysp$ is a locally trivial fibration with fiber $\bR$;
\item[\rm(b)]
the set $\Sigma$ of special elements of $\Xsp$ is locally finite;
\item[\rm(c)]
$\Ysp$ is a $T_1$-space locally homeomorphic with open subsets of $[0,1)$.
\end{enumerate}
Then every connected component $\Qsp$ of $\Xsp\setminus\Sigma$ is open in $\Xsp$ and is foliated homeomorphic with one of the following five stripped surfaces: model strips $\bR\times(0,1)$, $\bR\times[0,1)$, $\bR\times[0,1]$, or standard cylinder $C$, or standard M\"obius band $M$.
Moreover, in the last three cases, $\Qsp$ is also closed in $\Xsp$.
\end{lemma}
\begin{proof}
By (a) and Lemma~\ref{lm:loc_triv_relations} $\prj$ is an open map.
Therefore by Lemma~\ref{lm:special_elements} $\Sigma=\prj^{-1}(\YspecPtSet)$ and $\prj(\Sigma)=\YspecPtSet$, where $\YspecPtSet$ is the set of special points of $\Ysp$.
Then by (b) and~(\ref{enum:lm:prop:loc_finite_families}) of Lemma~\ref{lm:prop} $\YspecPtSet$ is also a locally finite family of points.
Due to (c) each point in $\Ysp$ is closed, whence $\YspecPtSet$ is closed in $\Ysp$.

Let $\Wsp$ be a connected component $\Ysp\setminus\YspecPtSet$.
Then $\Wsp$ is open in $\Ysp$ and open closed in $\Ysp\setminus\YspecPtSet$.
Therefore $\Qsp = \prj^{-1}(\Wsp)$ is open in $\Xsp$ and open closed in $\Xsp\setminus\Sigma$, i.e.\! $\Qsp$ is a connected component of $\Xsp\setminus\Sigma$.
Moreover, due to (a) the restriction $\prj:\Qsp \to\Wsp$ is a locally trivial fibration with fiber $\bR$, and by Lemma~\ref{lm:connected_components_of_nonspec_pts} $\Wsp$ is homeomorphic with one of the following spaces: $(0,1)$, $[0,1)$, $[0,1]$, $S^1$.
Therefore in the first three cases (when $\Wsp$ is contractible) $\Qsp$ is fiber-wise homeomorphic to a product $\bR\times\Wsp$, and in the last case, when $\Wsp\cong S^1$, $\Qsp$ is fiber-wise homeomorphic either with the standard cylinder $C$ or with the standard M\"obius band $M$.

It remains to show that every connected component $\Qsp$ of $\Xsp\setminus\Sigma$ can be represented as $\Qsp = \prj^{-1}(\Wsp)$ for some connected component $\Wsp$ of $\Ysp\setminus\YspecPtSet$.
Let $\Wsp = \prj(\Qsp)$.
We claim that $\Wsp$ is open closed in $\Ysp\setminus\YspecPtSet$.
Indeed, let $\Wsp'$ be the connected component of $\Ysp\setminus\YspecPtSet$ containing $\Wsp$.
Then as noted above $\prj^{-1}(\Wsp')$ is connected and contains $\Qsp$, whence $\Qsp = \prj^{-1}(\Wsp')$, and so $\Wsp = \Wsp'$.
\end{proof}

\section{Proof of~\eqref{th:open_strips:Q} of Theorem~\ref{th:open_strips}}\label{sect:proof:1:th:open_strips}
Let $\Xsp$ be a $2$-dimensional manifold and $\Partition$ be a $1$-dimensional foliation on $\Xsp$ belonging to class $\classFol$ and such that the set $\Sigma$ of special leaves of $\Xsp$ is locally finite.
Let also $\Ysp=\Xsp/\Partition$ be the space of leaves endowed with the corresponding factor topology and $\prj:\Xsp\to\Ysp$ be the factor map.

We claim that $\prj$ satisfies conditions (a)--(c) of Lemma~\ref{lm:local_prop}.
Indeed, by Lemma~\ref{lm:classFol_prop} $\prj$ is open, and by Lemma~\ref{lm:loc_triv_relations} it is a locally trivial fibration with fiber $\bR$, so condition (a) holds.
Condition (b) holds by assumption and condition (c) directly follows from definition of class $\classFol$.

Therefore by Lemma~\ref{lm:local_prop} every connected component $\Xsp\setminus\Sigma$ is foliated homeomorphic with one of the spaces: $\bR\times(0,1)$, $\bR\times[0,1)$, $\bR\times[0,1]$, $C$, $M$.

Applying the above result to the surface $\Xsp\setminus\partial\Xsp$ we get that every connected component of $\Xsp\setminus(\Sigma\cup\partial\Xsp)$ is foliated homeomorphic with one of the spaces: $\bR\times(0,1)$, $C$, or $M$.
Statement~\eqref{th:open_strips:Q} of Theorem~\ref{th:open_strips} is proved.

\section{Trapezoids}\label{sect:trapezoids}
The results of this section will be used for the proof of (2) of Theorem~\ref{th:open_strips}.

Let $c<d$ and $\alpha,\beta:(c,d] \to \bR$ be two continuous functions such that $\alpha(y) < \beta(y)$ for all $y\in(c,d]$.
Then the subset
\[\trap = \{ (x,y) \in \bR^2 \mid \alpha(y) \leq x \leq \beta(y), \ c< y\leq d \}\]
will be called a \emph{half open trapezoid} or simply a \emph{trapezoid}.
In this case $[\alpha(d), \beta(d)] \times d$ is the \emph{upper base} of $\trap$, $d$ is the \emph{level} of the upper base, $d-c$ is the \emph{altitude} of $\trap$, and the set
\[
\roof(\trap) := \{ (\alpha(y), y) \}_{y\in(c,d]}  \ \cup \ [\alpha(d), \beta(d)] \times d \ \cup \ \{ (\beta(y), y) \}_{y\in(c,d]}
\]
is the \emph{roof} of $\trap$, see Figure~\ref{fig:trapezoid}a).

\begin{figure}[ht]
\begin{tabular}{ccc}
\includegraphics[height=2.5cm]{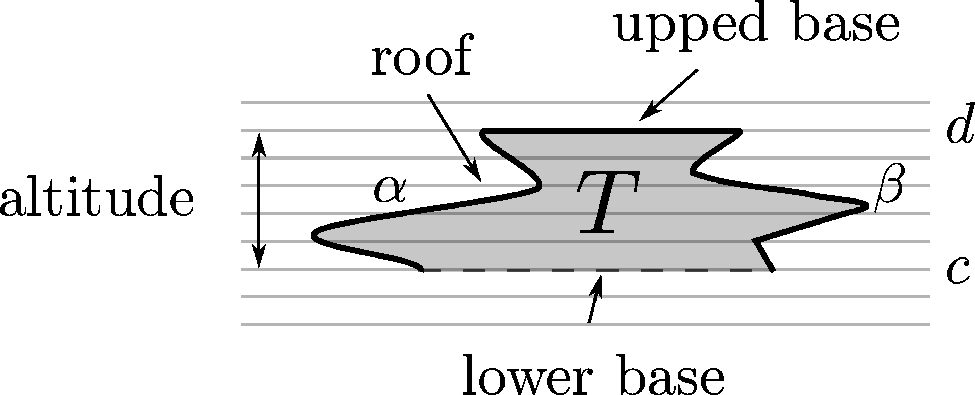} & \qquad \qquad &
\includegraphics[height=2.5cm]{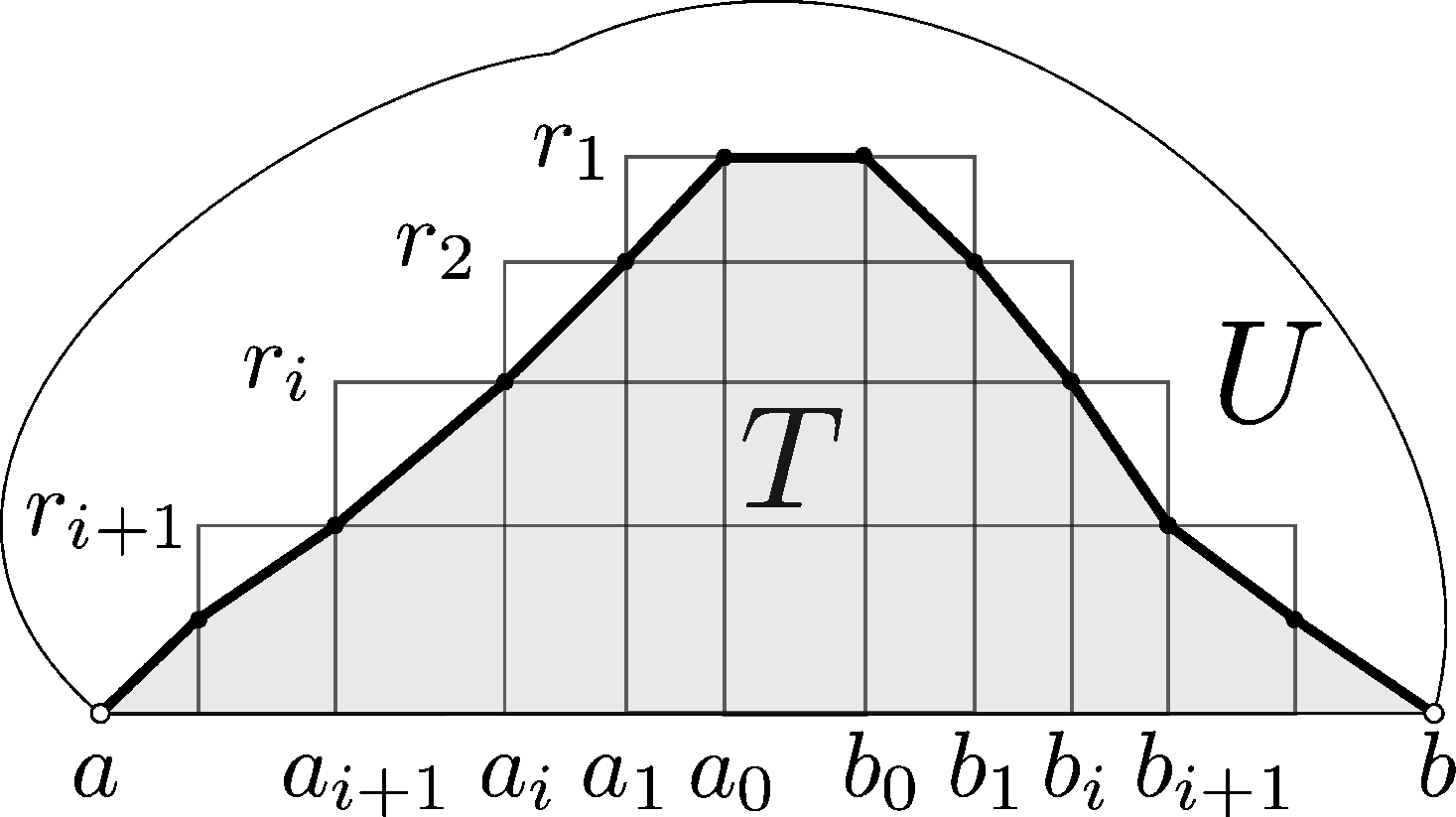} \\
a) Half open trapezoid & & b) Construction of trapezoid
\end{tabular}
\caption{Half open trapezoid}\label{fig:trapezoid}
\end{figure}

Notice that if $\phi:\bR\times(c,d] \to \bR\times(c,d]$ is a homeomorphism preserving second coordinate, i.e. $\phi(\bR\times y) =\bR\times y$ for all $y\in(c,d]$, then $\phi(\trap)$ is a trapezoid as well.

In general, $\alpha$ and $\beta$ can be non-bounded or have no limits when $y\to c+0$.
Suppose, in addition, that there exist finite or infinite limits
\begin{align*}
\lim\limits_{y\to c+0} \alpha(y) &= a,  &\lim\limits_{y\to c+0} \beta(y) = b
\end{align*}
such that $a<b$.
Then $(a,b)\times c$ will be called the \emph{(lower) base} of $\trap$.
If $a$ and $b$ are finite numbers, then $\trap$ will be called a \emph{trapezoid with bounded base}, and the set
\[\overline{\trap} = \trap \cup [a,b]\times c\]
will be a \emph{closed} trapezoid.
In particular, if $\alpha$ and $\beta$ are constant functions, then the trapezoid $\trap$ will be called a \emph{rectangle}.

\begin{lemma}\label{lemma:graph_neighborhood}
Let $J=(a,b)\times 0 \subset \bR^2$ be an open interval, $N = J \ \cup \ \bR^2 \times(0,+\infty)$, and $U$ be an open neighborhood of $J$ in $N$.
Then there exists a half open trapezoid $\trap \subset U$ with base $J$, see Figure~\ref{fig:trapezoid}b).
\end{lemma}
\begin{proof}
Fix any two sequences $\{a_i\}_{i=0}^{\infty}, \{b_i\}_{i=0}^{\infty} \subset (a,b)$ such that $\lim\limits_{i\to\infty} a_i= a$, $\lim\limits_{i\to\infty} b_i = b$, and
\[
\cdots < a_{i+1} < a_i < \cdots < a_0 < b_0 < \cdots < b_i < b_{i+1} < \cdots
\]
Let also $J_i = [a_i,b_i]\times 0$.
Since $U$ is an open neighborhood of $J_i$ and $J_i$ is compact, there exist $r_i>0$ such that $J_i \times [0,r_i] \subset U$.
One can assume that $\lim\limits_{i\to\infty} r_i = 0$ and $\{r_i\}$ is strictly decreasing.
Now let $\alpha, \beta:(0,d]\to(0,+\infty)$ be a unique continuous function such that for each $i\geq0$
\begin{itemize}
\item[(i)]
$\alpha(a_i) = \beta(b_i) = r_{i+1}$;
\item[(ii)]
the restrictions $\alpha|_{[a_{i+1},a_i]}$ and $\beta|_{[b_{i},b_{i+1}]}$ are linear.
\end{itemize}
Then one easily checks that the function $\alpha$ and $\beta$ are strictly monotone and their inverses $\alpha^{-1}$ and $\beta^{-1}$ determine a half open trapezoid $T\subset\Usp$ with base $J$.
\end{proof}

\begin{proposition}\label{prop:trap_make_rectangles}
Let $\trap_i \subset \bR\times(c,d]$, $i\in\bN$, be a half open trapezoid with upper base at level $d_i \in(c,d]$ such that $\trap_i\cap\trap_j = \varnothing$ for $i\not=j$ and $\lim\limits_{i \to \infty} d_i = c$.
Then there exists a homeomorphism $\eta:\bR\times(c,d]\to\bR\times(c,d]$ such that
\begin{itemize}
\item[(i)]
$\eta(\bR\times y) = \bR\times y$ for all $y\in(c,d]$;
\item[(ii)]
$\eta(\trap_i)$ is a half open rectangle.
\end{itemize}
\end{proposition}

\begin{proof}
We need the following three lemmas.
It will be convenient to say that for a function $f:[a,b]\to\bR$ its \emph{graph} is the subset $\{(f(y), y) \mid y\in [a,b]\} \ \subset \ \bR^2$, so we just swap coordinates with respect to the usual definition.
In particular, for $q\in\bR$ a vertical segment $q\times[a,b]$ can be regarded as a graph of a constant function $[a,b]\mapsto q$.

\begin{sublemma}\label{lm:funcction_uk}{\rm
{(c.f.~\cite[Lemma~6.1.1]{Maksymenko:BSM:2006}})}.
Let $\Delta_k = \{ (y_1,\ldots,y_k) \in \bR^k \mid y_1<y_2<\ldots<y_k \}$ and $q_1 < q_2 <\ldots < q_k \in \bR$.
Then there exists a $C^{\infty}$ function $u_k:\bR\times\Delta_k\to\bR$ having the following properties:
\begin{itemize}
\item[\rm(a)]
the correspondence $x \mapsto u_k(x; y_1,\ldots,y_k)$ is an orientation preserving homeomorphism $\bR\to\bR$ for all $(y_1,\ldots,y_k)\in\Delta_k$;
\item[\rm(b)]
$u_k(x; q_1,\ldots,q_k) = x$ for all $x\in\bR$;
\item[\rm(c)]
$u_k(y_i; y_1,\ldots,y_k) = q_i$ for $i=1,\ldots,k$.
\end{itemize}
\end{sublemma}
\begin{proof}
The construction of $u_k$ is similar to \cite[Lemma~6.1.1]{Maksymenko:BSM:2006}.
For instance, one can set
\begin{align*}
u_1(x; y_1) &= x - y_1 + q_1, &
u_2(x; y_1,y_2) &= q_1 + \frac{q_2-q_1}{y_2-y_1} (x - y_1).
\end{align*}
We leave the details for the reader.
\end{proof}

\begin{sublemma}\label{lm:rectification_finite}
Let $\func_i: (c, s] \to \bR$, $i = 1, \ldots, k$, be a finite family of continuous functions such that $\func_i(y) \not= \func_j(y)$ whenever $i\not=j$ and $y \in (c, s]$.
Then there exists a homeomorphism $\phi: \bR\times(c,s]\to \bR\times(c,s]$ such that
\begin{itemize}
\item[\rm(1)] $\phi(\bR\times y) = \bR\times y$ for all $y \in (c,s]$;
\item[\rm(2)] $\phi$ is fixed on $\bR\times s$;
\item[\rm(3)] $\phi$ maps the graph $\{(\func_i(y), y)  \mid y\in(c,s])\}$ of $\func_i$, $i \in \{1, \ldots, k\}$, onto a vertical segment $\func_i(s) \times [a, s]$.
\end{itemize}
\end{sublemma}
\begin{proof}
One can assume that $\func_i < \func_j$ for $i<j$.
Let $u_k:\bR\times\Delta_k\to\bR$ be a function from Lemma~\ref{lm:funcction_uk} constructed for the numbers $q_i = \func_i(s)$, $i \in \{1, \ldots, k\}$.
Then a homeomorphism $\phi$ satisfying (1)-(3) can be defined by the following formula:
\[
\phi(x,y) = \bigl( u_k(x; \func_1(y), \ldots,\func_k(y)), \ y \bigr).
\]
Indeed, (1) is evident.
Moreover, due to property (b) of $u_k$ we have that
\[
\phi(x,s)= \bigl( u_k(x; \func_1(s), \ldots,\func_k(s)), \ s \bigr) = (x,s)
\]
which proves (2).
Finally, by property (c) of $u_k$
\[
\phi(\func_i(y),y)=\bigl( u_k(\func_i(y); \func_1(y), \ldots,\func_k(y)), \ y \bigr) = (\func_i(s), y),
\]
so (3) is also satisfied.
\end{proof}

\begin{sublemma}\label{lemma:rectification_inifinite}
Let $\{\dd_i\}_{i\in\bN} \subset (c,d]$ be a sequence with $\lim\limits_{i \to \infty} d_i = c$, and for each $i\in\bN$ let $\func_i:(c, \dd_i]\to\bR$ be a continuous function such that the graphs of $\func_i$ and $\func_j$ are mutually disjoint for $i\not= j$.
Then there exists a homeomorphism $\eta:\bR\times(c,d]\to\bR\times(c,d]$ such that
\begin{itemize}
\item[(i)]
$\eta(\bR\times y) = \bR\times y$ for all $y\in(c,d]$;
\item[(ii)]
$\eta$ maps the graph $\{(\func_i(y), y) \mid y\in(c,\dd_i] \}$ of $\func_i$ onto a vertical segment $q_i \times(c,\dd_i]$ for some $q_i\in\bR$, $i\in\bN$.
\end{itemize}
\end{sublemma}
\begin{proof}
One can assume, in addition, that $\{\dd_i\}_{i\in\bN}$ is non-increasing.
Let us remove repeating elements from $\{\dd_i\}_{i\in\bN}$ and denote the obtained sequence by $\{s_i\}_{i\in\bN}$.
Thus there is an increasing sequence of indices $1 = j_1 < j_2 < \cdots < j_n < \cdots$ such that
\[s_i = d_{j_i} = d_{j_{i}+1} = \cdots = d_{j_{i+1} - 1} \ > \ s_{i+1} = d_{j_{i+1}} = \cdots . \]
Then by Lemma~\ref{lm:rectification_finite} there exists a homeomorphism $\phi_1:\bR\times(c,s_1]\to\bR\times(c,s_1]$ preserving second coordinate and sending the graphs of functions $\func_{j_1},\ldots,\func_{j_2-1}$ onto vertical segments.
Let us extend $\phi_1$ by the identity on $\bR\times[s_1,d]$ to a homeomorphism of all of $\bR\times(c,d]$.

Denote by $\func_i^1$ the image of the graph of $\func_i$ under $\phi_1$.
Then again there exists a homeomorphism $\phi_2:\bR\times(c,d]\to\bR\times(c,d]$ preserving second coordinate, fixed on $\bR\times[s_2,d]$, and sending the graphs of functions $\func^1_{j_1},\ldots,\func^1_{j_3-1}$ onto vertical segments.

Hence the composition $\phi_2\circ\phi_1$ preserves second coordinate and sends the graphs of functions $\func_{j_1},\ldots,\func_{j_3-1}$ onto vertical segments.
Denote by $\func^2_{i}$ the graph of the function $\func_i$ under $\phi_2\circ\phi_1$.

Then by similar arguments, we will construct an infinite family of homeomorphisms $\phi_1,\ldots,\phi_k,\ldots$ of $\bR\times(c,d]$ such that each $\phi_k$ preserves second coordinate, is fixed on $\bR\times[s_k,d]$, and sends the graphs of functions $\func^1_{j_1},\ldots,\func^1_{j_{k+1}-1}$ onto vertical segments.

Since $\lim\limits_{i\to\infty} s_i = c$, it follows that the infinite composition
\[
\eta = \cdots \circ \phi_m \circ\phi_{m-1} \circ \cdots \circ \phi_1:\bR\times(c,d]\to\bR\times(c,d]
\]
is a well defined homeomorphism satisfying the statement of lemma.
\end{proof}

To deduce Proposition~\ref{prop:trap_make_rectangles} assume that $\trap_i$ is defined by functions $\alpha_i,\beta_i:(c,d_i]\to\bR$.
Denote $\func_{2i-1} = \alpha_i$ and $\func_{2i} = \beta_i$.
Then existence of $\eta$ is guaranteed by Lemma~\ref{lemma:rectification_inifinite}.
\end{proof}

\subsection*{Level-preserving homeomorphisms between trapezoids.}
\sloppy
Let $q_2:\bR^2\to\bR$,    $q_2(x,y)=y$, be the standard projection onto the second coordinate and
\begin{align*}
\strap &= \{ (x,y) \in \bR \mid \alpha(y) \leq x \leq \beta(y), \ a < y \leq b \}, \\
\trap &= \{ (x,y) \in \bR \mid \gamma(y) \leq x \leq \delta(y), \ c < y \leq d \}
\end{align*}
be two trapezoids with finite bases, where $\alpha,\beta:(a,b]\to\bR$ and $\gamma,\delta:(c,d]\to\bR$ are continuous functions such that $\alpha<\beta$ and $\gamma<\delta$.

\fussy
Let $A \subset \strap$ and $B\subset\trap$ be two subsets.
Then a map $\xi: A \to B$ will be called \emph{level-preserving} whenever
\[
q_2\circ \xi(x,y) = q_2\circ \xi(x',y)
\]
for all $x,x',y$ such that $(x,y), (x',y)\in A$.

\begin{lemma}\label{lm:level_pres_homeo_of_roofs}
Every level-preserving homeomorphism $\xi:\roof(\strap) \to \roof(\trap)$ between roofs of trapezoids extends to a level-preserving homeomorphism $\xi:\strap \to \trap$.
Moreover, if $\strap$ and $\trap$ have finite bases, then $\xi$ also extends to a level-preserving homeomorphism $\xi:\overline{\strap} \to \overline{\trap}$ between their closures.
\end{lemma}
\begin{proof}
As $\xi$ is level-preserving, we have a well defined homeomorphism $\sigma:(a,b]\to(c,d]$ given by $\sigma(y) = q_2\circ \xi(\alpha(y),y)$.
Then $\xi$ extends to a homeomorphism $\strap\to\trap$ by
\[
\xi(x,y) = \left( \gamma(\sigma(y)) + \frac{\delta(\sigma(y))-\gamma(\sigma(y))}{\beta(y)-\alpha(y)} (x-\alpha(y)), \ \sigma(y) \right).
\]
Moreover, if in addition $\strap$ and $\trap$ have finite bases, so $\alpha$ and $\beta$ are defined and continuous on $[a,b]$ and $\gamma$ and $\delta$ are defined and continuous on $[c,d]$, then the same formulas define homeomorphisms $\sigma:[a,b]\to[c,d]$ and $\xi:\overline{\strap}\to\overline{\trap}$.
\end{proof}

\section{Proof of~\eqref{th:open_strips:closure_of_Q} of Theorem~\ref{th:open_strips}}\label{sect:proof:2:th:open_strips}
Let $\Partition$ be a partition on $\Xsp$ of class $\classFol$ such that the family $\Sigma$ of all special leaves is locally finite.
Let also $\bar{\Sigma} = \Sigma \cup \partial\Xsp$ be the union of all special and boundary leaves of $\Xsp$, $\Qsp$ be a connected component of $\Xsp\setminus\bar{\Sigma}$ homeomorphic with an open model strip, and $\phi:\bR\times(-1,1)\to\Qsp$ be a foliated homeomorphism.
Denote, see Figure~\ref{fig:q_properties}:
\begin{align}
\label{equ:notations_for_q}
\Qmin &= \phi\bigl(\bR\times(-1,0]\bigr), &
\Ksp &= \phi\bigl(\bR\times 0\bigr), &
\Qmax &= \phi\bigl(\bR\times[0,1)\bigr).
\end{align}
We should prove that the closures $\overline{\Qmin}$ and $\overline{\Qmax}$ are foliated homeomorphic to some model strips.
It suffices to prove this only for $\overline{\Qmin}$.

\begin{figure}[ht]
\includegraphics[height=1.4cm]{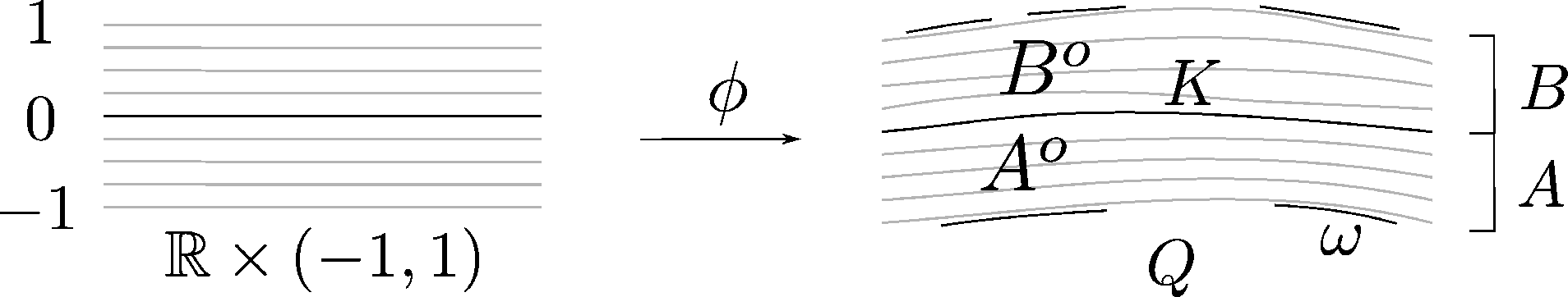}
\caption{}\label{fig:q_properties}
\end{figure}

\begin{lemma}\label{lm:Qmin_omega}
{\rm 1)}~$\overline{\Qmin} \setminus \Qsp = \overline{\Qmin} \setminus \Qmin$ \ and \ $\overline{\Qmax} \setminus \Qsp = \overline{\Qmax} \setminus \Qmax$.

{\rm 2)}~ $\overline{\Qsp} \setminus \Qsp= (\overline{\Qmin} \setminus \Qmin) \cup (\overline{\Qmax} \setminus \Qmax) \ \subset \ \bar{\Sigma}$.

{\rm 3)}~Let $\omega$ be a leaf in $\overline{\Qmin} \setminus \Qmin$, $J = (-1,1)\times 0 \subset \bR^2$, and $N = J \ \cup \ \bR \times (0, 1]$.
\begin{enumerate}
\item[\rm(a)]
Then $\Qsp \cup\omega$ is open in $\overline{\Qsp}$ and $\Qmin\cup\omega$ is open in $\overline{\Qmin}$.
\item[\rm(b)]
There exists a foliated homeomorphism $\psi:N \to \Qmin \cup \omega$ such that $\psi(J) = \omega$.
\item[\rm(c)]
Let $\Usp \subset \Xsp$ be an open neighborhood of $\omega$ and $T\subset \psi^{-1}(\Usp)$ be a subset with compact closure such that
\begin{align*}
&\overline{\Usp} \ \cap \ (\overline{\Qmin}\setminus\Qmin) \ = \ \omega, &
J \ \subset \ (\overline{T}\setminus T) \ \subset \ \bR\times 0.
\end{align*}
Then $\psi(T)$ is closed in $\Xsp$.
In particular, if $T$ is a trapezoid with base $J$, then $\psi(T \cup J)$ is closed in $\Xsp$.
\end{enumerate}
\end{lemma}
\begin{proof}
1)~Denote $\iQmin = \Qmin \setminus \Ksp$ and $\iQmax = \Qmax \setminus \Ksp$.
Then $K\subset \overline{\iQmin}$, so
\[\overline{\Qmin} = \overline{\iQmin \cup \Ksp} = \overline{\iQmin} \cup \Ksp =\overline{\iQmin}.\]
Moreover, as $\iQmin$ and $\iQmax$ are open in $\Xsp$ and disjoint, we get that $\overline{\Qmin} \cap \iQmax = \overline{\iQmin} \cap \iQmax = \varnothing$, whence $\overline{\Qmin} \setminus \Qsp = \overline{\Qmin} \setminus (\Qmin \cup \iQmax)  = \overline{\Qmin} \setminus \Qmin$.
The proof for $\Qmax$ is similar.

\smallskip

2)
It follows from (1) that $\overline{\Qsp} \setminus \Qsp = (\overline{\Qmin} \setminus \Qsp) \cup (\overline{\Qmax} \setminus \Qsp) = (\overline{\Qmin} \setminus \Qmin) \cup (\overline{\Qmax} \setminus \Qmax)$.

Let us prove that $\overline{\Qsp} \setminus \Qsp\subset\bar{\Sigma}$.
Suppose $\overline{\Qsp}\setminus\Qsp \not\subset \bar{\Sigma}$.
Then there exists a connected component $P$ of $\Xsp\setminus\bar{\Sigma}$ distinct from $\Qsp$ and such that $\overline{\Qsp}\cap P \not=\varnothing$.
But $P$ is open in $\Xsp$, whence $P\cap\Qsp\not=\varnothing$ and so $P=\Qsp$ which contradicts to the assumption.

\smallskip

3a)
Notice that the family $\bar{\Sigma}\setminus\{\omega\}$ is locally finite as well as $\bar{\Sigma}$.
Therefore the set
\[
\Wsp := \Xsp\setminus(\bar{\Sigma}\setminus\omega) = (\Xsp\setminus\bar{\Sigma})\cup\omega
\]
is open in $\Xsp$.
Due to 2), $\Qsp = \overline{\Qsp} \cap (\Xsp\setminus\bar{\Sigma})$, whence
$\Qsp \cup \omega =
\overline{\Qsp} \cap \bigl( (\Xsp\setminus\bar{\Sigma})\cup\omega \bigr) =
\overline{\Qsp} \cap \Wsp$
is open in $\overline{\Qsp}$.

Similarly, due to 1), $\Qmin=\overline{\Qmin} \cap \Qsp$, whence
$\Qmin\cup\omega =
 \overline{\Qmin} \cap  \overline{\Qsp} \cap \bigl( (\Xsp\setminus\bar{\Sigma})\cup\omega \bigr) =
\overline{\Qmin} \cap \Wsp$ is open in $\overline{\Qmin}$.

\smallskip

3b)
Notice that $\Qmin\cup\omega$ is saturated and by Lemma~\ref{lm:nbh_spec_points} $\prj(\Qmin\cup\omega)$ is homeomorphic with $[0,1]$.
Since $\prj:\Qmin\cup\omega\to\prj(\Qmin\cup\omega)$ is a locally trivial fibration with fiber $\bR$, we obtain that $\Qmin\cup\omega$ is foliated homeomorphic with $\bR\times[0,1]$ and therefore with $N$.

\smallskip

3c)
It suffices to prove that $\psi(T)$ is closed in $\overline{\Qmin}\setminus\Usp$ being a closed subset of $\Xsp$, which will imply that $\psi(T)$ is closed in $\Xsp$ as well.

Let $\{z_i\}_{i\in\bN} \subset \psi(T)$ be a sequence converging to some $z\in\overline{\Qmin}$.
We should prove that $z\in\psi(T)$ as well.
Let $(x_i,y_i)=\psi^{-1}(z_i) \in T$.
Since $\overline{T}$ is compact, one can assume that $\{(x_i,y_i)\}$ converges to some $(\bar{x},\bar{y}) \in \overline{T}$.

If $(\bar{x},\bar{y}) \in T$, then $z=\lim\limits_{i\to\infty} z_i = \lim\limits_{i\to\infty} \psi(x_i,y_i) = \psi(\bar{x},\bar{y}) \in \psi(T)$.
Otherwise, we have that $(\bar{x},\bar{y}) \in \overline{T}\setminus T \subset \bR\times 0$, so $\bar{y}=0$, and thus $\lim\limits_{i\to\infty} y_i = \bar{y} =0$.
This implies that $z\not\in\Qmin = \psi\bigl(\bR\times(0,1]\bigr)$.
Hence $z \in \overline{\Usp} \cap(\overline{\Qmin}\setminus\Qmin) =  \omega =  \psi(J) \subset \psi(T)$.
\end{proof}

Due to (\ref{enum:lm:prop:loc_finite_families_nbh}) of Lemma~\ref{lm:prop} there exists a family $\cU = \{U_\omega\}_{\omega \in \bar{\Sigma}}$ of neighborhoods of elements of $\bar{\Sigma}$ such that the closures of elements of $\cU$ are pairwise disjoint in $\Xsp$.

Let $\{\omega_i \}_{i\in\Lambda}$ be all the leaves contained in $\overline{\Qmin}\setminus\Qmin$.
Then $\Lambda$ is at most countable set and one can assume that either $\Lambda=\{1,\ldots,k\}$ for some finite $k$ or $\Lambda=\bN$.

By Lemma~\ref{lm:Qmin_omega} for each $i\in\Lambda$ there exists a foliated homeomorphism $\phi_i:N \to \Qmin \cup \omega_i$ such that $\psi_i(J) = \omega_i$.

Then $\psi_{i}^{-1}(U_{\omega_i})$ is an open neighborhood of $J = (-1,1)\times 0$, whence by Lemma~\ref{lemma:graph_neighborhood} there exists a trapezoid $\trap_{i} \,\subset\, \psi_{i}^{-1}(U_{\omega_i}) \, \cap \, \bR\times(0,1)$ with base $J$.
Put
\[
\widehat{\trap}_i = \trap_i \cup J.
\]
Then by Lemma~\ref{lm:Qmin_omega} $\psi_i(\widehat{\trap}_i)$ is closed in $\overline{\Qmin}$.

\begin{figure}[ht]
\includegraphics[height=2.9cm]{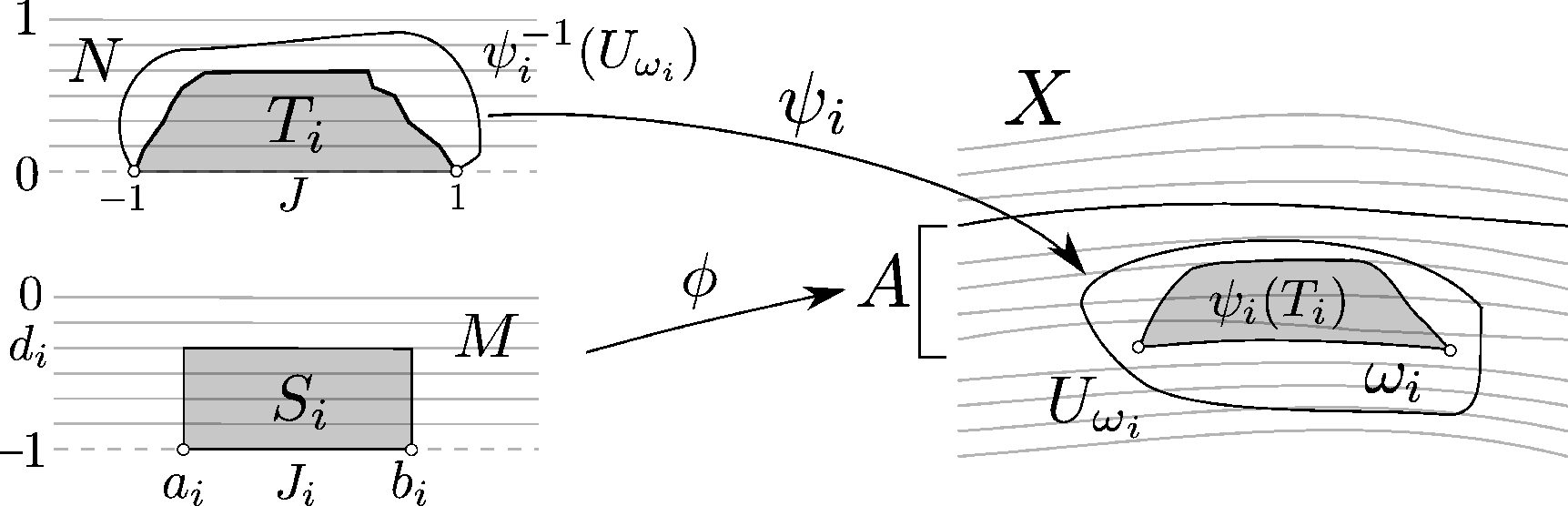}
\caption{}\label{fig:proof_2}
\end{figure}

Denote $\strap_{i} = \phi^{-1} \circ \psi_{i}(\trap_{i})$.
Then $\{\strap_{i} \mid i\in\Lambda\}$ is a family of trapezoids in $\bR\times(-1,0]$.
Assume that the upper base of $\strap_i$ in contained in $\bR\times \dd_i$ for some $\dd_i\in (-1,0)$.
If $\Lambda$ is infinite, then decreasing, if necessary, $\trap_{i}$ (and therefore $\strap_i$), one can assume that $\lim\limits_{i\to\infty} \dd_i = -1$.
Then by Proposition~\ref{prop:trap_make_rectangles} one can change $\phi$ so that
$\strap_i = [a_i, b_i] \times (-1, d_i]$ is a ``half open rectangle'' for some $a_i,b_i\in\bR$.
Then $[a_i,b_i] \cap [a_j,b_j]=\varnothing$ for $i\not=j\in\Lambda$.
Let also $J_i = (a_i,b_i) \times \{-1\}$ and
\[
M \ := \ \bR\times(-1,0] \ \bigcup \ \mathop{\cup}\limits_{i\in\Lambda} J_i.
\]

Then $M$ is a \emph{half model strip}.
Our aim is to construct a foliated homeomorphism between $M$ and $\overline{\Qmin}$.

Denote $\widehat{\strap}_i = \strap_i \cup J_i$, $i\in\Lambda$, and
\[
\Zsp \, := \, M \,\setminus\, \mathop{\cup}\limits_{i\in\Lambda} \bigl(\widehat{\strap}_i\setminus\roof(\strap_i)\bigr) \ \subset \ \bR\times(0,1].
\]
\begin{lemma}\label{lm:hat_strapi_is_loc_finite}
$\{\Zsp\}\cup \{\widehat{\strap}_i\}_{i\in\Lambda}$ is a locally finite cover of $M$ by closed sets.
\end{lemma}
\begin{proof}
It is evident, that $\widehat{\strap}_i$ is closed in $M$.
Moreover $\widehat{\strap}_i\setminus\roof(\strap_i)$ is open in $M$, whence $\Zsp$ is closed in $M$ as well.
Therefore it remains only to show that each $z=(x,y)\in M$ has an open neighborhood $\Vsp$ intersecting only finitely many elements $\widehat{\strap}_i$.

If $y=-1$, then $z\in(a_i,b_i)\times\{-1\} \subset \widehat{\strap}_i$ for some $i\in\Lambda$.
Hence $\Vsp=\widehat{\strap}_i\setminus\roof(\strap_i)$ is an open neighborhood of $z$ in $N$ intersecting only $\widehat{\strap}_i$.

Suppose that $y>-1$.
Fix any $t$ such that $-1<t<y$.
Then $\Vsp = \bR\times(t,0]$ is an open neighborhood of $z$ in $M$.
By assumption $\lim\limits_{i\to\infty}d_i = -1$, whence there exists $n>0$ such that $-1<d_i<t$ for all $i>n$, and so $\widehat{\strap}_i\cap\Vsp=\varnothing$.
\end{proof}

\begin{lemma}\label{lm:psii_trap_is_loc_finite}
$\{\phi(\Zsp)\}\cup \{\psi_i(\widehat{\trap}_i)\}_{i\in\Lambda}$ is a locally finite cover of $\overline{\Qmin}$ by closed sets.
\end{lemma}
\begin{proof}
By 3c) of Lemma~\ref{lm:Qmin_omega} each $\psi_i(\widehat{\trap}_i)$ is closed in $\Xsp$.
Furthermore,
\begin{align*}
\phi(\Zsp)
&= \phi\Bigl( M \,\setminus\, \mathop{\cup}\limits_{i\in\Lambda}\, \bigl( \widehat{\strap}_i\setminus\roof(\strap_i) \bigr) \Bigr)
= \overline{\Qmin} \,\setminus\, \mathop{\cup}\limits_{i\in\Lambda}\, \psi_i\bigl(\widehat{\trap}_i\setminus\roof(\trap_i)\bigr),
\end{align*}
and it is also evident that $\widehat{\trap}_i\setminus\roof(\trap_i)$ is open in $N$.
But due to 3b) of Lemma~\ref{lm:Qmin_omega} $\psi_i$ is a homeomorphism of $N$ onto the open subset $\Qmin\cup\omega_i$ of $\overline{\Qmin}$.
Therefore $\psi_i(\widehat{\trap}_i\setminus\roof(\trap_i))$ is open in $\overline{\Qmin}$, whence $\phi(\Zsp)$ is closed in $\overline{\Qmin}$.

It remains to show that $\{\psi_i(\widehat{\trap}_i)\}_{i\in\Lambda}$ is a locally finite family.
Let $q\in\overline{\Qmin}$.

If $q\in\omega_i$, then $\Usp_{\omega_i}$ is an open neighborhood of $q$ intersecting only one set $\psi_i(\widehat{\trap}_i)$.

Suppose $q\in\overline{\Qmin}\setminus\Qmin$ and let $z=(x,y) = \phi^{-1}(q) \in \bR\times(-1,0] \subset M$.
Then by Lemma~\ref{lm:hat_strapi_is_loc_finite} there exists an open neighborhood $\Vsp$ of $z$ in $\bR\times(-1,0]$ intersecting only finitely many $\widehat{\strap}_i$.
But the map $\phi:\bR\times(-1,0] \to \Qmin$ is a homeomorphism, whence $\phi(\Vsp)$ is an open neighborhood of $q$ in $\Qmin$ intersecting only finitely many $\psi_i(\widehat{\trap}_i) = \phi(\strap_i) \cup\omega_i$.
\end{proof}

Notice that the composition $\psi^{-1}\circ\phi|_{\strap_i}: \strap_i\to\trap_i$ is a level-preserving homeomorphism, however in general it can not be extended to a homeomorphism between their bases.
Nevertheless, $\psi^{-1}\circ\phi$ yields a level-preserving homeomorphism $\roof(\strap_i)\to \roof(\trap_i)$, and therefore by Lemma~\ref{lm:level_pres_homeo_of_roofs} it extends to a level-preserving homeomorphism $\xi_i:\overline{\strap_i} \to \overline{\trap_i}$.

Now define the following map $\eta:M \to \overline{\Qmin}$ by
\[
\eta(z) =
\begin{cases}
\psi_{i} \circ \xi_i(z), & z\in \widehat{\strap}_i \ \text{for some } i\in\Lambda, \\
\phi(z), & z\in\Zsp.
\end{cases}
\]

We claim that $\eta$ is the required homeomorphism.

Indeed, evidently, $\eta$ is a bijection.
Furthermore, by Lemma~\ref{lm:hat_strapi_is_loc_finite} $\{\Zsp\}\cup \{\widehat{\strap}_i\}_{i\in\Lambda}$ is a locally finite closed cover of $M$, and by Lemma~\ref{lm:psii_trap_is_loc_finite} their images $\{\phi(\Zsp)\} \cup \{\psi_i(\widehat{\trap}_i)\}_{i\in\Lambda}$ constitute a locally finite closed cover of $\overline{\Qmin}$.
Finally, the restrictions $\eta|_{\Zsp}$ and $\eta|_{\widehat{\strap}_i}$ are homeomorphisms.
Then by~(\ref{enum:lm:prop:homeomorphism}) of Lemma~\ref{lm:prop} $\eta$ is a homeomorphism.
Theorem~\ref{th:open_strips} is completed.





\begin{thebibliography}{100000}


\bibitem{AndronovPontryagin:DANSSSR:1937}
A.~Andronov and L.~Pontryagin , \emph{Syst$\grave{e}$mes grossiers}, Dokl. Akad. Nauk.
  SSSR \textbf{14} (1937), 247--251. (Russian)

\bibitem{AransonGrines:MatSb:1973}
S.~H. Aranson and V.~Z. Grines, \emph{On some invariants of dynamical systems on two-dimensional manifolds
(necessary and sufficient conditions for the topological equivalence of transitive dynamical systems)},
  Math. USSR Sb. \textbf{19} (1973), no. 3, 365--394.


\bibitem{AransonGrines:UMN:1986}
S.~Kh. Aranson and V.~Z. Grines, \emph{Topological classification of flows on
  closed two-dimensional manifolds}, Russ. Math. Surv.  \textbf{41} (1986),
  no.~1, 183--208. 

\bibitem{AransonGrinesKaimanovich:JDCS:2003}
S.~Kh. Aranson, V.~Z. Grines, and V.~A. Kaimanovich, \emph{Classification of
  supertransitive 2-webs on surfaces}, J. Dynam. Control Systems \textbf{9}
  (2003), no.~4, 455--468. 

\bibitem{AransonZhuzhomaMedvedev:MatSb:1997}
S.~Kh. Aranson, E.~V. Zhuzhoma, and V.~S. Medvedev, \emph{On continuity of
  geodesic frameworks of flows on surfaces}, Sb. Math. \textbf{188} (1997),
  no.~7, 955--972. 


\bibitem{Boothby:AJM_1:1951}
William~M. Boothby, \emph{The topology of regular curve families with multiple
  saddle points}, Amer. J. Math. \textbf{73} (1951), 405--438. 

\bibitem{Boothby:AJM_2:1951}
William~M. Boothby, \emph{The topology of the level curves of harmonic
  functions with critical points}, Amer. J. Math. \textbf{73} (1951), 512--538.

\bibitem{BronsteinNikolayev:TA:1997}
Idel Bronstein and Igor Nikolaev, \emph{Peixoto graphs of {M}orse-{S}male
  foliations on surfaces}, Topology Appl. \textbf{77} (1997), no.~1, 19--36.

\bibitem{BudnytskaPryshlyak:UMJ:2009}
N.~V. Budnyts'ka and O.~O. Pryshlyak, \emph{Equivalence of closed 1-forms on
  surfaces with edge}, Ukrainian Math. J. \textbf{61} (2009), no.~11,
  1710--1727. 


\bibitem{BudnytskaRybalkina:UMJ:2012}
N.~V. Budnyts'ka and T.~V. Rybalkina, \emph{Realization of a closed 1-form on
  closed oriented surfaces}, Ukrainian Math. J. \textbf{64} (2012), no.~6,
  844--856. 


\bibitem{Farber:AMSP:2004}
Michael Farber, \emph{Topology of closed one-forms}, Mathematical Surveys and
  Monographs, vol. 108, American Mathematical Society, Providence, RI, 2004.

\bibitem{JenkinsMorse:AJM:1952}
James~A. Jenkins and Marston Morse, \emph{Contour equivalent pseudoharmonic
  functions and pseudoconjugates}, Amer. J. Math. \textbf{74} (1952), 23--51.

\bibitem{Kaplan:DJM:1940}
Wilfred Kaplan, \emph{Regular curve-families filling the plane, {I}}, Duke
  Math. J. \textbf{7} (1940), 154--185. 

\bibitem{Maksymenko:BSM:2006}
Sergey Maksymenko, \emph{Stabilizers and orbits of smooth functions}, Bull.
  Sci. Math. \textbf{130} (2006), no.~4, 279--311. 

\bibitem{MaksymenkoPolulyakh:PGC:2015}
Sergiy Maksymenko and Eugene Polulyakh, \emph{Foliations with non-compact
  leaves on surfaces}, Proc. Intern. Geom. Center \textbf{8} (2015), no.~3--4,
  17--30.

\bibitem{Morse:FM:1952}
M.~Morse, \emph{The existence of pseudoconjugates on {R}iemann surfaces}, Fund.
  Math. \textbf{39} (1952), 269--287 (1953). 

\bibitem{OshemkovSharko:MatSb:1998}
A.~A. Oshemkov and V.~V. Sharko, \emph{On the classification of {M}orse-{S}male
flows on two-dimensional manifolds}, Sb. Math. \textbf{189} (1998), no.~8,
  1205--1250. 

\bibitem{Peixoto:Top:1962}
M.~M. Peixoto, \emph{Structural stability on two-dimensional manifolds},
  Topology \textbf{1} (1962), 101--120. 

\bibitem{Peixoto:Top:1963}
M.~M. Peixoto, \emph{Structural stability on two-dimensional manifolds. {A}
  further remark}, Topology \textbf{2} (1963), 179--180. 

\bibitem{Plachta:fol2:MMFMP:2001}
L.~P. Plachta, \emph{The combinatorics of gradient-like flows and foliations on
  closed surfaces. {II}. {T}he problem of realization and some estimates}, Mat.
  Metodi Fiz.-Mekh. Polya \textbf{44} (2001), no.~2, 7--16. 

\bibitem{Plachta:fol3:MMFMP:2001}
L.~P. Plachta, \emph{The combinatorics of gradient-like flows and foliations on
  closed surfaces. {III}. {T}he problem of realization and some estimates},
  Mat. Metodi Fiz.-Mekh. Polya \textbf{44} (2001), no.~3, 7--16. 

\bibitem{Plachta:fol1:TA:2003}
Leonid Plachta, \emph{The combinatorics of gradient-like flows and foliations
  on closed surfaces. {I}. {T}opological classification}, Topology Appl.
  \textbf{128} (2003), no.~1, 63--91. 


\end{thebibliography}
\end{document}